\documentclass{article}

\usepackage{arxiv}

\usepackage[utf8]{inputenc} 
\usepackage[T1]{fontenc}    
\usepackage{hyperref,bm}       
\usepackage{url}            
\usepackage{booktabs}       
\usepackage{amsfonts}       
\usepackage{nicefrac}       
\usepackage{microtype}      
\usepackage{lipsum}

\usepackage{color}
\usepackage{xcolor}
\usepackage{soul}

\usepackage{amsmath
            ,bbm}

\usepackage{amsthm}

\newcommand{\bP}{\mathbb{P}}
\newcommand{\bR}{\mathbb{R}}
\newcommand{\bS}{\mathbb{S}}

\newcommand{\cF}{\mathcal{F}}

\newcommand{\bfX}{\mathbf{X}}

\usepackage[square,numbers,sort&compress]{natbib}

\theoremstyle{definition}
\newtheorem{definition}{Definition}[section]

\newtheorem{lemma}{Lemma}[section]
\newtheorem{thm}{Theorem}[section]

\theoremstyle{remark}
\newtheorem{remark}{Remark}[section]

\usepackage{mathrsfs,eucal,dsfont}

\newcommand{\I}{\mathds 1}
\newcommand{\R}{\mathds R}

\newcommand{\1}{\mathbbm{1}}

\def\d{{\rm d}}
\def\<{\langle}
\def\>{\rangle}

 \def\ss{\sqrt}
\def\bb{\beta}

\newcommand{\bZ}{\mathbb{Z}}
\def\R{\mathbb R}   \def\ss{\sqrt} 
  
  \def\vv{\varepsilon} 
\def\<{\langle} \def\>{\rangle}  
  \def\nn{\nabla}  
\def\d{\text{\rm{d}}} \def\bb{\beta}  
  \def\si{\sigma} 
 \def\beq{\begin{equation}}  \def\F{\mathscr F}
 
\def\e{\text{\rm{e}}}  \def\OO{\Omega}  
  
 \def\P{\mathbb P}

  \def\ll{\lambda}
 
\def\E{\mathbb E}

\def\to{\rightarrow}
\def\8{\infty}\def\3{\triangle}
\def\1{\lesssim}

\renewcommand{\bar}{\overline}
\renewcommand{\hat}{\widehat}
\renewcommand{\tilde}{\widetilde}

\title{The random periodic solutions for McKean-Vlasov stochastic differential equations
}

\author{
 Jianhai Bao \\
  Center for Applied Mathematics\\
  Tianjin University\\
 300072 Tianjin, P.R. China\\ 
  \texttt{jianhaibao@tju.edu.cn}  \\
   \AND
   Gon\c calo Dos Reis \\
  School of Mathematics, University of Edinburgh, Peter Guthrie Tait Road,  EH9 3FD, UK\\ 
  and\\
Center for Mathematics and Applications (NOVA Math), 2829-516 Caparica, PT \\
  \texttt{G.dosReis@ed.ac.uk}  \\
    \AND
  Yue Wu \\
  Department of Mathematics and Statistics\\
  University of Strathclyde\\
 Glasgow,  G1 1XH, UK\\ 
  \texttt{yue.wu@strath.ac.uk}   
}

\usepackage{pifont}
\begin{document}
\maketitle

\begin{abstract}
In this paper,  we study well-posedness of random periodic solutions of stochastic  differential equations (SDEs) of McKean-Vlasov type driven by a two-sided Brownian motion, where the random periodic behaviour is characterised by the equations' long-time behaviour. Given the well-known connection between McKean-Vlasov SDEs and interacting particle systems,  we show  propagation of chaos and that the key properties of the interacting particle systems recover those of the McKean-Vlasov SDEs in the particle limit.  All results in the present work are shown under two settings: fully and partially dissipative case. Each setting has its challenges and limitations.  For instance, weakening full dissipativity to partial dissipativity demands stronger structural assumptions on the equations' dynamics and yields random periodic behaviour in the weak sense instead of pathwise sense (as in the full dissipativity case). 
The proof mechanisms are close but fundamentally different.
\end{abstract}



\newpage
\section{Introduction}

The invariant probability measure (PIM) has been well studied in the field of stochastic analysis \cite{da1996ergodicity} as it provides a way to understand the long-time behaviour of natural random phenomena  in the distribution sense. 
For more details on existence and uniqueness of IPMs for Markov processes, we refer to the excellent lecture notes \cite{hairer2006ergodic}. In the pathwise sense, it has been proved that,  under sufficient conditions,  a stationary trajectory of an SDE  admits a unique IPM \cite{arnold1995random}. Recently,  IPM was even adopted in the machine learning context for approximating the minimal weights of parameters in neural networks \cite{bottou2018optimization,chen2024mean}. Periodic measures (PMs), which can be thought of as a generalisation of the IPM, is also important and sometimes affects  ergodicity of the system under consideration. For instance, for a one-dimensional probabilistic cellular automation (PCA), a periodic measure of the period two may appear which leads to a unique but non-ergodic IPM  \cite{chassaing2011non}. However, there was a lack of literature addressing   existence/properties of periodic measures in a random dynamical system until very recently \cite{feng2020random}, after the introduction of the concept of the \textit{random periodic path} and \textit{pathwise random periodic solution} of cocycles/semiflows \cite{zhao2009random,feng2011pathwise}. These concepts generalise the definition of both the \textit{stationary trajectory} and the \textit{deterministic periodic trajectory} \cite{zhao2009random,wu2021random} and allow rigorous lens to understand physically interesting problems of certain random phenomena with (pseudo-)periodic patterns \cite{chekroun2011stochastic}. In recent works e.g. \cite{feng2011pathwise,feng2012random,feng2016anticipating,wu2018random},    existence of random periodic solutions for stochastic dynamical systems generated by  SDEs  and stochastic partial differential equations (SPDEs) was addressed.  For such works, a uniformly dissipative condition was generally imposed on the drift coefficient to guarantee its uniqueness \cite{feng2017numerical,wu2021random,wu2021galerkin}. In the past few years, the uniformly dissipative condition   was relaxed to be 
the counterpart which is 
dissipative in the  long distance or dissipative on average \cite{bao2022random}. 
    
An interesting bidirectional relation between the random periodic trajectory and the IPM was revealed in \cite{feng2020random}: the law of the random periodic trajectory is shown to be periodically preserved by some map which, in turn, provides a non-trivial characterisation of the statistical behaviour of the system; conversely, from a periodic measure one can construct an enlarged probability space and a random periodic process whose law is the periodic measure (this is akin to the Lions lift \cite{carmona2018probabilistic}). This fact provides a strategy to study random periodic trajectory of S(P)DEs and its law,  one after the other. It is however not capable of the class of McKean-Vlasov SDEs because  such an SDE has its law involved in its drift and diffusion coefficients:
\begin{equation}\label{E1}
\d X_{s,t}=b_t(X_{s,t},\mathscr L_{X_{s,t}})\,\d t+\si_t(X_{s,t},\mathscr L_{X_{s,t}})\,\d W_t, \quad t\ge s\in\R
\end{equation}with the initial condition $X_{s,s}=\xi$,  an integrable and $\cF_s$-measurable $\xi$, where  the maps 
$$b:\R\times\R^d\times\mathscr P_2(\R^d)\to\R^d,\qquad \sigma:\R\times\R^d\times\mathscr P_2(\R^d)\to\R^d\otimes\R^m$$
are measurable.  
One has to handle both of the random periodic solution and its law simultaneously. In this work, we will tackle the challenge and study \textit{random periodic solutions}  for the class of McKean-Vlasov SDEs\footnote{In the literature, due to the preferred interpretation of the processes involved, the McKean-Vlasov SDE is also named as the nonlinear SDE (e.g. \cite{kolokoltsov2010nonlinear}) or the mean-field SDE (e.g. \cite{carmona2018probabilistic}) or the  distribution-dependent SDE (e.g. \cite{huang2021distribution}). In some occasions, they are  used interchangeably. }.

McKean-Vlasov SDEs  have been intensively studied in statistical physics and have recently garnered renewed interest since they are the natural settings for stochastic mean-field games \cite{carmona2018probabilistic}. Nonetheless, their modelling scope is much broader than just mean-field games. Studying the long-time behaviour of such complex systems is of significant importance, especially when this system is proven as the limit of the single motion within models for complex phenomena involving a large number of interrelated components \cite{sznitman1991topics,Meleard1996,AdamsSalkeld2022LDPExitTimes,chen2023wellposedness}.  Our work brings the concept of random periodic solutions to the class of McKean-Vlasov SDEs  and paves a way for new modelling capabilities in fields, where the underlying physical system has periodic features that need to be captured and the McKean-Vlasov SDE is a suitable model. For instance, in neuroscience mean-field models \cite{LocherbachMonmarche2022Metastability} including the modelling of time-periodic features \cite{van2014periodic,FriedrichKinzel2009}, metastability of particle systems and their limits \cite{LocherbachMonmarche2022Metastability}, preservation of modelling properties across particle limit and scaling laws  \cite{chapman2012coarsening} in materials sciences, seasonal dynamics underlying electricity markets \cite{Alasseur2020storageSmartGrids}.

Owing to the involvement of measure variables in the coefficients, there are major discrepancies on  distributional properties between classical SDEs and McKean-Vlasov SDEs. For instance, the solution processes of classical SDEs are related to linear Markov processes whereas the McKean-Vlasov SDEs correspond  to nonlinear Markov  processes \cite{kolokoltsov2010nonlinear};   distributions of solution processes associated with classical SDEs solve  linear Fokker-Planck equations nevertheless the counterparts of McKean-Vlasov SDEs solve  nonlinear Fokker-Planck equations. As far as a classical SDE is concerned, the corresponding  Markov operator is a linear semigroup while the one associated with a McKean-Vlasov SDE is no longer a semigroup (see e.g.~ \cite{wang2018distribution,huang2021distribution}). As for classical SDEs, the distribution properties of solution processes with a general initial distribution can be investigated thoroughly once the counterpart starting from the Dirac measure is available (i.e., the flow property in $\bR^d$), while the point of view above is not any longer valid (see, for example, \cite{wang2018distribution} for further details). Furthermore, concerning classical SDEs, the localization procedure is powerful to handle e.g., existence and uniqueness of strong/weak solutions, ergodic properties, numerical methods, to name a few. Unfortunately, such a strategy  cannot  carry over (at least directly)  to  the setting of McKean–Vlasov SDEs; see e.g., \cite{reis2018simulation} for more details.

\textit{Our contributions.} 
In this paper, our first goal is to establish  a comprehensive framework to understand \textit{random periodic trajectory}  for SDEs of McKean-Vlasov type. 
As a more involved system, since the equation's dynamics depends on the law of the solution process itself, the task is not straightforward. 
There are several manners to tackle  the issue on existence and uniqueness of strong solutions to the McKean-Vlasov SDE  \eqref{E1}; see, for instance, \cite{salkeld2019LDPMV-SDEs,wang2018distribution} via a distribution Picard-iteration procedure and \cite{huang2021distribution} by invoking the Banach fixed point theorem. 
Furthermore, we refer to a comprehensive survey paper \cite{huang2021distribution} on recent progresses concerning McKean-Vlasov SDE theory, and the monograph \cite{carmona2018probabilistic}. 
As aforementioned, by its nature,  existence and uniqueness of  random periodic solutions and periodic measures are established simultaneously in our work (rather than one   after another as seen in all previous literature).

McKean-Vlasov SDEs are often obtained as limits of interacting particle systems (IPSs)  and the corresponding mathematical tool that shows  convergence of the IPS to the original McKean-Vlasov SDE is termed as the \emph{propagation of chaos} (PoC) \cite{sznitman1991topics,Meleard1996,AdamsSalkeld2022LDPExitTimes}. 
Our second contribution is the study of the behaviour around the infinite time horizon and we carry it out from two perspectives. One is via direct analysis of the McKean-Vlasov SDE and making use of the nonlinear flow  defined on the space of probability measure (see   Theorem  \ref{th1}). 
The other is via the particle limit of the IPS system. In other words, we firstly study the long-time behaviour of the IPS system and then leverage  PoC type results to ensure the long-time behaviour of the IPS carrying across the particle limit (see Theorems\ref{thm2} and   \ref{thm3} ). This exchange of limits result is far from trivial, and we point to the well-known   fact that invariant distributions of the IPS do not necessarily carry over as invariant distributions to the McKean-Vlasov SDE \cite{HerrmannTugaut2010}.

In the last few decades, there have been great advances in the study of existence of time-periodic laws for non-linear diffusion processes and associated non-linear Fokker-Planck  equations. 
Particularly,  \cite{scheutzow1985some}   provided a collection of non-trivial examples on the topic: a distribution-dependent Ornstein-Uhlenbeck process, a deterministic example, a one-dimensional example with a semi-linear drift, and a two-dimensional example with a linear drift. 
Additionally, \cite{scheutzow1986periodic} revealed that the stochastic Brusselator in the mean-field limit has a solution with a periodic law (with respect to time). Further recent progresses on periodic behaviours of mean-field models and associated particle systems can be found in \cite{cormier2021hopf,dai2020oscillatory,luccon2020emergence,Marini2023etalfrustratednetworks} and references within.  In \cite{Marini2023etalfrustratednetworks}, the authors focused on a particular interacting particle model and investigated mechanisms driving the emergence of periodic behaviours including in the corresponding mean-field limit. They combined both the effect of noise and the topology of the interaction network to study a model of so-called frustrated interacting diffusions and showed noise-induced periodicity. Their scopes were fairly broad as they studied the interacting particle system, propagation of chaos, the limiting mean-field model, the small noise limit case (and its ODE system) and behaviour from the associated non-linear Vlasov-Fokker-Plank equation including the accompanying numerical studies.
\color{black}
As for a long view, our manuscript proposes a new concept to be used for modelling and opening door for further studies: for instance, the vanishing noise behaviour by means of large deviations \cite{salkeld2019LDPMV-SDEs,HIPP2014Resonance}, numerical simulation \cite{reis2018simulation,wu2021random,Chen2024SuperSUper}, McKean-Vlasov SDE with jumps and random periodicity within the more complex  mean-field setup with common noise.

The paper is organised in the following way: in Section \ref{sec:framework}, the main notations are introduced, including ${X}^{\xi}_{s,t}$ the solution at time $t$ of the McKean-Vlasov SDE \eqref{E1} starting from  $\xi$ at  time $s$, ${X}^{*}_{t}$ the random periodic solution at time $t$ of the McKean-Vlasov equation, $ {\bf X}^{\xi^N}_{s,t}$ the solution at time $t$ of the corresponding IPS starting from $\xi^N$ at time $s$, and $ {\bf X}^{*,N}_{t}$ the random periodic solution at time $t$ of the corresponding IPS, and two main results, under the uniformly dissipative condition and the partially dissipative condition respectively, are summarised; Section \ref{sec:ProofofMVSDEwellposednesstheo:th1} is devoted to the proof of the first main result, i.e., the red downward arrow on the right side of the chart, showing existence and uniqueness of  random periodic solutions and periodic measures for  McKean-Vlasov equations; Section \ref{sec:ProofofInteractingParticSystem:PoC:thM2} is concerned with   the proof of the second main result, i.e., the blue double-arrow at the bottom of the chart, showing the PoC carrying over the results under the infinite time horizon.   
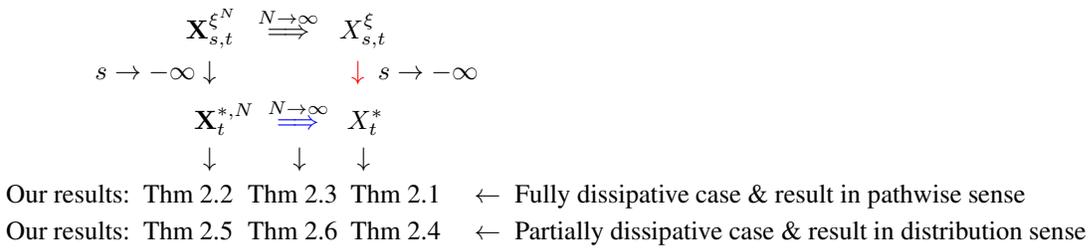
\begin{align*}
    &\hspace{1cm}  \ {\bf X}^{\xi^N}_{s,t}\ \ \overset{N\to \infty}{\Longrightarrow}\ \ {X}^{\xi}_{s,t}
    \\
    &\!\! {s\to -\infty}\downarrow \ \ \ \ \  \hspace{1.26cm}{\color{red}\downarrow}\ \ {s\to -\infty }
    \\
    &\hspace{1.2cm}\mathbf {X}^{*,N}_t \ \ \overset{N\to \infty}{{\color{blue}\Longrightarrow}} \ \ \hspace{0.07cm}X^{*}_t 
    \\
    & \hspace{1.22cm} \downarrow \ \ \hspace{0.75cm}{\downarrow}\ \ \hspace{0.50cm}{\downarrow}
    \\
    & \hspace{-1.3cm}\textrm{Our results:}  \ \ \textrm{Thm \ref{thm2}} \hspace{0.20cm}{\textrm{Thm \ref{thm3}}}\ \ {\textrm{Thm \ref{th1}}} 
   \quad \leftarrow \textrm{ Fully dissipative case \& result in  pathwise sense }
    \\
    & \hspace{-1.3cm}\textrm{Our results:}  \ \ \textrm{Thm \ref{thm4}} \hspace{0.20cm}{\textrm{Thm \ref{thm6}}}\ \ {\textrm{Thm \ref{thm*}}} 
   \quad \leftarrow \textrm{ Partially dissipative case \& result in distribution sense}
\end{align*}

The two assumption archetypes we explore lead to different formulations of the main results, where in one formulation the results hold in a pathwise sense while in the second formulation they hold in a distributional sense. The first block of assumptions focuses on a general dissipative condition and proofs follow more intuitive arguments. The second block of assumptions addresses the partially dissipative case but requires further structural assumptions on the McKean-Vlasov  dynamics and the proof makes use of coupling arguments.

Under the fully dissipative condition, the synchronous coupling approach is employed to treat   existence and uniqueness of pathwise random periodic solutions; see, for instance,  \cite{feng2017numerical,wu2021random,wu2021galerkin}. Whereas, the synchronous coupling trick does not work once the uniformly dissipative condition is broken. 
In terms of \cite[Proposition 2.1]{bao2022random}, ergodicity under the Wasserstein distance plays a crucial role in investigating existence and uniqueness of random periodic solutions in the weak sense. So far,  one of the powerful tools in exploring ergodic properties for stochastic dynamical systems is 
the coupling method, which has been   investigated considerably in dealing with e.g. eigenvalue estimate and spectral gap 
in the monograph \cite{chen2004markov}.   Recently, a great breakthrough on exponential ergodicity for SDEs with non-uniformly dissipative coefficients  is due to \cite{eberle2016reflection}. 
In the spirit of \cite{eberle2016reflection},
more progresses on exponential ergodicity for non-uniformly dissipative SDEs driven by Brownian motions or pure jump L\'{e}vy processes have been made by invoking the reflection coupling or the refined basic coupling; see e.g. \cite{luo2019refined,ren2021exponential,wang2021exponential}.
For non-dissipative  McKean-Vlasov SDEs,  exponential ergodicity under the Wasserstein-type distance has been considered in \cite{wang2021exponential}. Subsequently,  \cite{wang2021exponential} was extended to the setting of time-periodic McKean-Vlasov SDEs, where  exponential ergodicity  in relative entropy was also discussed via the Talagrand equality and the log-Sobolev inequality. In Section 5, we also address  exponential ergodicity for the McKean-Vlasov SDE \eqref{E1} and the corresponding interacting particle system under the partially dissipative condition. With contrast to  \cite{ren2021exponential,wang2021exponential}, a much more explicit transportation cost function is designed in the present paper. Most importantly,  via a much more direct approach (compared with the splicing method adopted in \cite{ren2021exponential,wang2021exponential}), 
we cope with the PoC 
uniform  in time and limit interchangeability for time-periodic McKean-Vlasov SDEs with multiplicative noise, which was not involved \cite{ren2021exponential,wang2021exponential}.

In the past few years, there is a huge amount of literature concerned with PoC in a  finite time horizon; see, for instance, \cite{sznitman1991topics,carmona2018probabilistic}
 for McKean-Vlasov SDEs with regular coefficients. 
Since the pioneer work \cite{durmus2020elementary} on investigating by the coupling approach
the uniform-in-time  PoC 
 for time-homogeneous granular media equations  with additive noise,  the issue on the uniform-in-time  PoC 
 has received more and more attention; see, for example, \cite{lacker2023sharp,guillin2022convergence} for McKean-Vlasov SDEs with additive noises.

\section{Framework and main results} \label{sec:framework}

\subsection*{Notations and spaces}

  

Let $(\R^d,\<\cdot,\cdot\>,|\cdot|)$ be the $d$-dimensional Euclidean space with the inner product $\<\cdot,\cdot\>$ inducing the Euclidean norm $|\cdot|$ and $(\R^{m}\otimes\R^d,\<\cdot,\cdot\>_{\rm HS},\|\cdot\|_{\rm HS})$ the collection of all $m\times d$ matrices with real entries, which is endowed with the Hilbert-Schmidt product $\<\cdot,\cdot\>_{\rm HS}$ inducing the Hilbert-Schmidt norm $\|\cdot\|_{\rm HS}$.
Let $\bZ$ be the space of integer numbers and, for each integer $N\ge1$, define $\mathbb S_N=\{ 1,2,\cdots,N\}$. For $x\in \R^d$ or $A\in \R^{d\times m}$,  define $x^*$ and $A^*$ as their corresponding transposes. For $x,y\in \R^d$, define the function $\text{vec}$ iteratively as $\text{vec}(x,y):=(x^*,y^*)^*$, which is a column vector. We set $\bf0$ as the origin of $\bR^d$ (the dimension is implied if no confusion arises).

Set $L_{\rm loc}(\R;\R_+)$ as the space of locally integrable functions $f:\R\to \R_+$. 
Denote $\mathscr P(\mathbb{R}^d)$ by the class
of probability measures on $\mathbb{R}^d$. For $p>0$,  
$\mathscr P_p(\mathbb{R}^d)$  refers to  
a subspace of $\mathscr P(\R^d)$ with finite $p$-th moment, i.e., 
\begin{equation*}
  \mathscr P_p(\mathbb{R}^d)=\bigg\{\mu\in\mathscr P(\R^d)\Big|\ \mu(|\cdot|^p):=\int_{\R^d}|x|^p\,\mu(\d x)<\8\bigg\}, 
\end{equation*}
 which is a Polish space under the Wasserstein distance
\begin{align*}
    \mathbb{W}_p(\mu,\nu)=\inf_{\pi\in\mathcal C(\mu,\nu)}\Bigg\{\bigg(\int_{\mathbb{R}^d\times\mathbb{R}^d}\mathrm{d}(x,y)^p\pi(\mathrm{d}x,\mathrm{d}y)\bigg)^{\frac{1}{1\vee p}} \Bigg\},
\end{align*}
where $\mathscr C(\mu,\nu)$ means the set of couplings of $\mu$ and $\nu$, i.e., $\pi\in\mathscr C(\mu,\nu)$ if and only if $\pi(\cdot,\R^d)=\mu(\cdot)$ and $\pi(\R^d,\cdot)=\nu(\cdot)$.
For a random variable $\xi$,    we shall write $\xi\sim\mu$ to represent  that $\xi$ is distributed according to $\mu$. For $x\in\R^d$,   $\delta_x$ denotes the Dirac delta measure centered at $x$. For $x,y\in \R^d$,  $\delta_{(x,y)}$ specifies the Dirac delta measure centered at $\text{vec}(x,y)$.

Let $\mathbb{P}$ be the two-sided $d$-dimensional Wiener measure on $(\Omega, \mathscr{F})$, which is a measure-preserving probability. Write  $(W(t))_{t \in \mathbb{R}}$ as the $d$-dimensional Brownian motion on the probability space $(\Omega, \mathscr{F}, \mathbb{P})$; see \cite{bao2022random} for more details. We denote by $L^p(\OO\to\R^d,\F_s,\P)$ the space of $\F_s$-measurable random variables $X$ in $\R^d$ with finite $p$-th moment, i.e., $\E|X|^p<\8$.

Let $\theta$ be the Wiener shift operator defined by $\left(\theta_{t} \omega\right)(s)=\omega(t+s)-\omega(t)$ for all $s, t \in \mathbb{R}$ and $\omega \in \Omega$. Then, $(\Omega, \mathscr{F}, \mathbb{P}, \theta)$ is a metric dynamical system. 
For each $\omega \in \Omega$ and $t \in \mathbb{R}$, define $W(t, \omega)=\omega(t)$. We will use $W_t$ for $W(t,\omega)$ when there is no ambiguity about $\omega$.
 
For a vector valued (or matrix-valued) function $f$ on $\mathbb{R}$ and a positive constant $\tau, f$ is said to be $\tau$-periodic if $f(t+\tau)=f(t)$ for all $t \in \mathbb{R}$. Set $\triangle:=\{(s,t)\in\R^2: s\le t, t,s\in\R\}$. For two random variables 
$\xi,\eta$, we write $\xi\overset{d}{=}\eta$ to demonstrate that $\xi$ and $\eta$
share the same law. 


\subsection*{Framework}
To guarantee  existence and uniqueness of strong solutions to the McKean-Vlasov SDE \eqref{E1}, for the drift term $b$ and the diffusion term $\sigma$,  we assume that 
\begin{enumerate}
\item[(\bf A)] 
$\R\ni t\mapsto |b_t({\bf0},\delta_{{\bf0}})|^2+\|\sigma_t({\bf0},\delta_{{\bf0}})\|^2_{\rm HS}$  is continuous, and, for each $t\in\R$, $ b_t$ is continuous and bounded on bounded sets of $\R^d\times \mathscr P_2(\R^d)$. 
 Moreover, there exist $\bar{K}_1,\bar{K}_2\in L_{\rm loc}(\R;\R_+)$ such that   for all $t\in\R,$ $x,y\in\R^d$ and $\mu,\nu\in\mathscr P_2(\R^d)$, 
\begin{equation}\label{E2}
2\<x-y,b_t(x,\mu)-b_t(y,\nu)\>\le \bar{K}_1(t)\big(|x-y|^2+\mathbb W_2(\mu,\nu)^2\big),
\end{equation}
and 
\begin{equation}\label{E3}
\|\sigma_t(x,\mu)-\sigma_t(y,\nu)\|_{\rm HS}^2\le \bar{K}_2(t)\big(|x-y|^2+\mathbb W_2(\mu,\nu)^2\big).
\end{equation}
\end{enumerate}

In terms of  \cite[Theorem 3.3]{huang2021distribution}, \eqref{E1} has a unique strong solution $(X_{s,t})_{t\ge s}$ and, for a given finite horizon $T(\ge s)$, there exists a constant $C=C(s,T)>0$ such that 
\begin{equation*}
\E\Big(\sup_{s\le t\le T}|X_{s,t}|^2\Big)\le C\big(1+\E|X_{s,s}|^2\big)   
\end{equation*}
for all
 $X_{s,s}\in L^2(\OO\to\R^d,\F_s,\P)$. Below, in some occasions,   we shall write $(X_{s,t}^\xi)_{t\ge s}$ instead of $(X_{s,t})_{t\ge s}$ to emphasise the initial value $X_{s,s}=\xi\in L^2(\OO\to\R^d,\F_s,\P)$. 

As we know,   Assumption ({\bf A}) dominates the short-time behaviour (e.g. existence and uniqueness of solutions, continuous dependency on initial distributions and time regularity) of solutions to the McKean-Vlasov SDE \eqref{E1}. Whereas,   Assumption ({\bf A}) is insufficient to manipulate  the long-term behaviour (e.g., existence of random periodic solution \cite{ren2021exponential} and ergodic property \cite{wang2021exponential}) of solution processes to \eqref{E1}.

Before we proceed, let's retrospect the notions on random periodic solutions for a semi-flow involved.   Below, the underlying state space $H$ is set to be   a separable Banach space.

\begin{definition}
\label{feng-zhao1}(\cite[Definition 1.1]{feng2011pathwise})
   For a fixed $\tau>0$, an ${\cal F}$-measurable
map $X^*:\mathbb{R}\times \Omega\to H $ is called a pathwise random $\tau$-periodic solution of the  semi-flow  $\phi: \Delta\times H\times \Omega\to H$ if for all $(s,t)\in\triangle$ and a.s. $\omega\in \Omega$,
\begin{eqnarray*}
\phi(t,s, X^*_s(\omega), \omega)=X^*_t(\omega) \quad \mbox{ and } \quad
X^*_{t+\tau}(\omega)=X^*_t( \theta_\tau \omega).
\end{eqnarray*}
\end{definition}

\begin{remark}
In  Definition \ref{feng-zhao1}, 
the initial time point and initial condition are both hidden. The initial time point is  set to be $-\infty$ by default thus it may be unnecessary to give the initial condition; see  \cite{wu2021random,bao2022random} for a detailed  discussion.
\end{remark}

The study on  existence of pathwise random periodic solutions has been treated in variant settings.  Herein, we would like to refer to     \cite{feng2011pathwise,feng2016anticipating,feng2017numerical,wu2021random} for $\mathbb{R}^d$-valued SDEs with $H=\mathbb{R}^d$, \cite{feng2012random,wu2018random} 
for  SPDEs on a bounded domain $D \subset \mathbb{R}^d$ with $H=L^2(D)$, and    \cite{bao2022random} for a functional  SDE with a finite time lag $r_0$ and $H=L^2\big(\Omega \to C([-r_0,0];\mathbb{R}^d)\big)$.  
In some scenarios, the pathwise random periodic solution is unavailable for a semi-flow under investigation. In this setting, we shall focus on a weak notion on random periodic solutions, that is, the random periodic solution in the sense of distribution, which is stated as follows.

\begin{definition}(\cite[Theorem 1.1]{bao2022random})
\label{def:random-period-solution} For a fixed number $\tau>0,$
an ${\cal F}$-measurable
map $X^*:\mathbb{R}\times \Omega\to H$ is called a  random $\tau$-periodic solution in the sense of distribution of the  semi-flow  $\phi: \Delta\times H\times \Omega\to H$ if for all $(s,t)\in\triangle,$
\begin{eqnarray*}
\phi(t,s, X^*_s(\omega), \omega)\overset{d}{=}X^*_t(\omega) \quad \mbox{ and } \quad 
X^*_{t+\tau}(\omega)\overset{d}{=}X^*_t( \theta_\tau \omega) .
\end{eqnarray*}

\end{definition}
Under a partially dissipative condition,   in \cite{bao2022random} it is set $H=L^1(\Omega\to \mathbb{R}^d)$  to study random periodicity of an   $\mathbb{R}^d$-valued SDE with an additive noise.

%

\subsection*{Main result  under a general dissipative condition}
Regarding the McKean-Vlasov SDE \eqref{E1}, to investigate existence of random periodic solutions (which definitely is concerned with the long-term behaviour of solutions), we need to strengthen  Assumption ({\bf A}) by  Assumption ({\bf A'}), as stated below. 

\begin{enumerate}
\item[(\bf A')] Assume {(\bf A)}. In addition, for each $x\in\R^d$ and each $\mu\in\mathscr P_2(\R^d)$, $t\mapsto b_t(x,\mu)$ and $t\mapsto \si_t(x,\mu)$ are $\tau$-periodic for some deterministic constant    $\tau>0$. Moreover, there exist $\tau$-periodic functions $K_1\in C(\R;\R)$ and $K_2,K_3\in C(\R;\R_+)$ such that for all $t\in\R$, $x,y\in\R^d$ and $\mu,\nu\in\mathscr P_2(\R^d)$, 
\begin{equation} \label{E20}
2\<x-y,b_t(x,\mu)-b_t(y,\nu)\>+\|\sigma_t(x,\mu)-\sigma_t(y,\nu)\|_{\rm HS}^2\le K_1(t)|x-y|^2+K_2(t)\mathbb W_2(\mu,\nu)^2,
\end{equation}
and 
\begin{equation} \label{E21}
\|\sigma_t(x,\mu)-\sigma_t(y,\nu)\|_{\rm HS}^2\le K_3(t)\big(|x-y|^2+\mathbb W_2(\mu,\nu)^2\big).
\end{equation}
\end{enumerate}

{In the following, an illustrative example is constructed  to show that \eqref{E20} is available.  Let $\R\ni t\mapsto K_1(t)\in[-1,\frac{1}{2}]$ be a $1$-periodic function such that 
\begin{equation*}
K_1(t)=
 \begin{cases}
   -2t,\quad \quad\quad ~~0\le t\le \frac{1}{2},\\
   6t-4,\quad \quad~~\frac{1}{2}\le t\le \frac{3}{4},\\
   -2(t-1), \quad \frac{3}{4}\le t\le 1.
 \end{cases}   
\end{equation*}
Then, we define for any $x\in\R$ and $\mu\in\mathscr P(\R)$, 
\begin{equation*}
b_t(x,\mu)=K_1(t)\big((x^3+x)\I_{\{0\le t\le \frac{3}{4}\}}+x\I_{\{\frac{3}{4}\le t\le 1\}}
\big)+\frac{1}{2}K_2(t)\mu(|\cdot|),
\end{equation*}
and $\si_t(x,\mu)=\frac{1}{2}K_2(t)\mu(|\cdot|)$, where $\R\ni t\mapsto K_2(t)\in[0,\8)$ is a $1$-periodic function. A direct calculation shows that \eqref{E20} is satisfied. 
}

Our first main result is concerned with existence of pathwise random periodic solutions to the McKean-Vlasov SDE \eqref{E1}. The proof is postponed to  Section \ref{sec:ProofofMVSDEwellposednesstheo:th1}.

\begin{thm}\label{th1}
Assume   $({\bf A'})$ with \begin{equation}\label{E23}
-\lambda:=\int_0^\tau\big(K_1(u)+K_2(u)\big)\,\d u\in(-\8,0).
\end{equation} 
Then, 
the stochastic semi-flow $\phi$ defined by
\begin{equation}\label{E22}
\Delta\times L^2(\Omega\to\R^d,\F_s,\P) \times \Omega    \ni (s,t,\xi,\omega)\mapsto \phi(t,s,\xi,\omega)=X_{s,t}^\xi(\omega),
 \end{equation}
where  $X_{s,t}^\xi(\omega)$ solves  \eqref{E1},  possesses a unique  \textit{pathwise random $\tau$-periodic solution}. That is, 
 there exists a unique stochastic process $ X^{*}_t \in L^2(\Omega\to\R^d,\F_t,\P)$ such that  for all $t\in\R$ and $h\ge0$,  
 \begin{equation*}
X^*_{t+h}(\omega)=\phi(t+h, t, X^*_t(\omega),\omega)\quad \mbox{ and } \quad X^*_{t+\tau}(\omega)=X^*_t(\theta_\tau\omega)\quad {\rm a.s. }
\end{equation*}
Moreover, for all $t\in\R$  and $ \xi \in L^2(\Omega\to\R^d,\F_s,\P)$,
\begin{equation} \label{P17}
    \lim_{s\downarrow-\8}\E|X^{\xi}_{s,t}-X^*_t|^2=0.
\end{equation} 
\end{thm}

A common element for McKean-Vlasov equations is their associated (interacting and non-interacting) particle systems and PoC results.  Consider the auxiliary non-interacting particle system  associated with \eqref{E1}: for  the integer $N\geq 1$,  $i\in \mathbb S_N $, and $(s,t)\in\Delta,$
\begin{equation}\label{P1}
\d X_{s,t}^{i}=b_t(X_{s,t}^{i},\mu_{s,t}^i)\,\d t+\si_t(X_{s,t}^{i},\mu_{s,t}^i)\,\d W_t^i,  \quad   \quad \mathscr L_{X_{s,s}^i}=\mathscr L_{X_{s,s}}=\mu,
\end{equation}
where $(X_{s,s}^i,W_t^{i})_{i\in \mathbb S_N}$   are i.i.d. copies of $(X_{s,s},W_t)$ and  $\mu_{s,t}^i:=\mathscr L_{X_{s,t}^{i}}$. Due to the strong well-posedness of \eqref{E1}, we obviously have $\mu_{s,t}^i=\mu_{s,t}=\mathscr L_{X_{s,t}}$ with $\mathscr L_{X_{s,s}}=\mu,$ where $(X_{s,t})_{t\ge s}$ is the solution process to \eqref{E1}. 
Replacing  $\mu_{s,t}^i$ in \eqref{P1} by the empirical measure generated by the $N$ particles, we obtain the IPS related to \eqref{E1}: for all  $i\in \mathbb S_N $, and $(s,t)\in\Delta$,
\begin{equation}\label{P2}
\d X_{s,t}^{i,N}
=
b_t\big(X_{s,t}^{i,N},\hat{\mu}^{{\bf X}^N}_{s,t}\big)\,\d t
+\si_t\big(X_{s,t}^{i,N},\hat{\mu}^{{\bf X}^N}_{s,t}\big)\,\d W_t^i, 
\quad   \quad  X_{s,s}^{i,N}=X_{s,s}^i,
\end{equation}
where the empirical measure $\hat{\mu}^{{\bf X}^{N,N}}_{s,t}$   
is defined by
$$\hat{\mu}^{{\bf X}^{N,N}}_{s,t}(\d z)=\hat{\mu}^{{\bf X}_{s,t}^{N,N}}(\d z):=\frac{1}{N}\sum_{j=1}^N\delta_{X_{s,t}^{j,N}}(\d z),$$
with ${\bf X}_{s,t}^{N,N}:=\text{vec}\big(X_{s,t}^{1,N},\cdots,X_{s,t}^{N,N}\big)\in( \R^{d})^N$.

Set
$$\xi^N:=\text{vec}(\xi_1,\cdots,\xi_N)\quad  \mbox{ and } \quad   \eta^N:=\text{vec}(\eta_1,\cdots,\eta_N),$$
where 
  $\xi_i, \eta_i\in L^p(\OO\to\R^d,\F_s,\P)$, $i\in\mathbb S_N$, are  i.i.d copies of $X_{s,s}\in L^p(\OO\to\R^d,\F_s,\P)$ for some $p\ge2$. In the sequel,  we shall also write 
${\bf X}_{s,t}^{N,N,\xi^N}$ and ${\bf X}_{s,t}^{N,N,\eta^N}$ instead of ${\bf X}_{s,t}^{N}$ in case of ${\bf X}_{s,s}^{N,N}=\xi^N$ and ${\bf X}_{s,s}^{N,N}=\eta^N$
to   highlight the initial value.  
For each fixed integer $N\ge1$, define the mapping $\phi^N:\triangle\times(\R^{d})^N\times\Omega\to(\R^{d})^N$ by 
\begin{equation}\label{1T}
(s,t,\xi^N,\omega)\mapsto\phi^N(t,s,\xi^N,\omega)={\bf X}_{s,t}^{N,N,\xi^N}(\omega).
\end{equation}

The next results confirms the well-posedness of the particle system and the pathwise random periodic solution property.
\begin{thm}\label{thm2}
Under  assumptions of Theorem \ref{th1}, for each fixed $N\ge1$,
the stochastic semi-flow $\phi^N$  defined in  \eqref{1T}, possesses a unique pathwise random $\tau$-periodic solution. That is, 
there exists a unique stochastic process $ {\bf X}^{*,N,N}_t =\text{vec}\big(X_{t}^{*,1,N},\cdots,X_{t}^{*,N,N}\big)\in L^2(\Omega\to(\R^{d})^N,\F_t,\P)$ such that  for all $t\in\R$ and $h\ge0$,
 \begin{equation}\label{P14}
{\bf X}^{*,N,N}_{t+h}(\omega)=\phi^N(t+h, t, {\bf X}^{*,N,N}_t(\omega),\omega),\qquad {\bf X}^{*,N,N}_{t+\tau}(\omega)={\bf X}^{*,N,N}_t(\theta_\tau\omega)\quad {\rm a.s. }
\end{equation}
Furthermore,  for all $t\in\R$ and $ \xi^N \in L^2(\Omega\to(\R^{d})^N,\F_s,\P)$,
\begin{equation} \label{P15}
    \lim_{s\downarrow-\8}\E|{\bf X}^{N,N,\xi^N}_{s,t}-{\bf X}^{*,N,N}_t|^2=0.
\end{equation}
\end{thm}
The proof is postponed to Section \ref{sec:ProofofInteractingParticSystem:PoC:thM2}.

We close with a result connecting the particle systems concerned with the McKean-Vlasov SDE under investigation. The first statement is classical saying that the IPS \eqref{P2} recovers, in the particle limit, the McKean-Vlasov SDE \eqref{P1}. The second statement, is novel and states that    the random periodic
solution associated with the IPS recovers, in the particle limit, that of the McKean-Vlasov SDE \eqref{P1}.
\begin{thm}[Uniform-in-time PoC]
\label{thm3}
Assume that  assumptions  of Theorem \ref{th1}   hold. 
Then, there exists a constant $C>0$ (independent of particle number $N$ and uniformly bounded in initial time $s$) such that for all $t\in\R$ and $i\ge1$,
\begin{equation}\label{P19} \E|X_{s,t}^{i}-{X}_{s,t}^{i,N}|^2
+ \E|X_t^{*,i}-{X}_t^{*,i,N}|^2\le C\varphi(N)
\to 0 \ \textrm{ as }\ N\to\infty,
\end{equation}
where 
$(X^i_{s,t})_{i\in \bS_N}$ is the unique solution of the non-interacting particle system \eqref{P1},  $(X^{*,i})_{i\in \bS_N}$ is the unique random periodic solution corresponding to  \eqref{P1}, 
$({ X}_{s,t}^{i,N})_{i\in \bS_N}$ is the unique solution associated with the interacting particle system \eqref{P2}, 
$({ X}_t^{*,i,N})_{i\in \bS_N}$ is the unique random periodic solution associated ${\bf X}^{*,N,N}$ related   to \eqref{P2}, and 
\begin{equation}\label{P7}
\varphi(N)=
\begin{cases}  
 N^{-\frac{1}{2}} +N^{-\frac{\vv_0}{2+\vv_0}},\qquad\qquad \qquad~~~ d<4,\\
 N^{-\frac{1}{2}}\log(1+N)+N^{-\frac{\vv_0}{2+\vv_0}},\qquad d=4,\\
 N^{-\frac{2}{d}}+N^{-\frac{\vv_0}{2+\vv_0}},~~~\qquad\qquad \qquad d>4,
\end{cases}
\end{equation}
where the precise value of $\vv_0\in(0,1)$ can be traced by a close inspection of the proof for Lemma
\ref{lem3} in Section $3.$

\end{thm}
The proof is postponed to Section \ref{sec:ProofofInteractingParticSystem:PoC:thM2}.

\subsection*{Main results under a partial dissipative condition}

The results presented in the previous section are of general nature under the dissipative condition \eqref{E20}. We now show that our results still  hold true when  \eqref{E20} is replaced  by a  partial dissipativity condition (see \eqref{X3} below) but at the expense of additional structural conditions on the dynamics of the McKean-Vlasov SDE. The proofs we present later on draw on coupling arguments (see e.g. \cite{wang2021exponential}) which is a fundamentally different approach with contrast to that  used to establish the previous results.

We now work with McKean-Vlasov SDEs of the following type (see Remark \ref{remark why the label{*X} equation has that shape} for more details):  
\begin{equation}\label{*X}
\d X_{s,t}=\big(\,\hat b_t(X_{s,t})+(\tilde b_t*\mathscr L_{X_{s,t}})(X_{s,t})\big)\,\d t+\sqrt{\alpha_t}\d B_t+\hat \si_t(X_{s,t})\,\d W_t, \quad (t,s)\in\triangle,
\end{equation}
where $\hat{b}:\R\times \R^d\to\R^d$, for any $x\in \bR^d$ and $\mu \in \mathscr P_1(\bR^d)$,  $(\tilde{b}_t*\mu)(x):=\int_{\R^d}\tilde{b}_t(x,y)\,\mu(\d y)$ with 
$\tilde{b}:\R\times \R^d\times\R^d\to\R^d$,
$\hat \sigma:\R\times\R^d\to\R^d\otimes\R^d$, $\alpha:\R\to\R_+$, 
and $(W_t)$ and $(B_t)$ are two independent two-sided Brownian motions.

\begin{remark}
\label{remark why the label{*X} equation has that shape}
Note that two independent identically distributed noises $(B_t)$ and $(W_t)$ are involved in \eqref{*X}, which  makes the formulation \eqref{*X} non-standard. Furthermore, we would like to stress that  \eqref{*X} is not a  McKean-Vlasov SDE with common noise since the distribution involved is deterministic (rather than random). Indeed, instead of \eqref{*X}, this paper could have focused on the more general McKean-Vlasov SDE: 
\begin{equation}\label{*X*}
\d X_{s,t}=\big(\,\hat b_t(X_{s,t})+(\tilde b_t*\mathscr L_{X_{s,t}})(X_{s,t})\big)\,\d t+\si_t(X_{s,t})\,\d W_t, \quad (t,s)\in\triangle,
\end{equation}
where $(\hat b,\tilde b, W)$ would be as in \eqref{*X} and $\si:\R\times \R^d\to\R^d\otimes\R^d$ is measurable. As long as  the diffusion term $\si$ satisfies a uniform ellipticity condition, one can decompose (in the sense of distribution) the term $\si_t(X_{s,t})\,\d W_t$ into two parts: one part resembles ``$\sqrt{\alpha_t}\d B_t$'' and the other one ``$\hat \si_t(X_{s,t})\,\d W_t $''. 
Based on this point of view, we prefer the McKean-Vlasov SDE \eqref{*X} in lieu of \eqref{*X*}. Most importantly, the formulation \eqref{*X} makes it easier to deploy an asymptotic reflection coupling for the additive noise, and a synchronous coupling for the multiplicative counterpart (see   Remark \ref{remark on couplings} for further details). 
\end{remark}

Assume that 
\begin{enumerate}
\item[(\bf H)] Fix $\tau \in (0,\infty)$. Let $\R\ni t\mapsto \hat b_t({\bf0})$, $\R\ni t\mapsto \hat \sigma_t({\bf0})$, $\R\ni t\mapsto\tilde b_t({\bf0},{\bf0})$ and $\R\ni t\mapsto\alpha_t$ are $\tau$-periodic and there exist
constants $K_0,\ell_0\ge0, K_1,K_2,K_3>0$ such that for all $t\in\R$ and $x,y,\tilde x,\tilde y\in\R^d$, 
\begin{align}\label{X3}
\nonumber
\<x-y,\hat b_t(x)-\hat b_t(y)\>
& +\frac{1}{2}\|\hat \sigma_t(x)-\hat \sigma_t(y)\|_{\rm HS}^2
\\
&\le \alpha_t\big(K_0|x-y|^2\I_{\{|x-y|\le\ell_0\}}-K_1|x-y|^2\I_{\{|x-y|>\ell_0\}}\big), 
\end{align}
and
\begin{equation}\label{XX}
|\tilde{b}_t(x,y)-\tilde{b}_t(\tilde x,\tilde y)|\le K_2\alpha_t\big(|x-\tilde x|+|y-\tilde y|\big),\quad \|\hat \sigma_t(x)-\hat \sigma_t(y) \|_{\rm HS}^2\le K_3\alpha_t|x-y|^2,
\end{equation}
where $\I_\cdot$ means  the indicator function.
\end{enumerate}

Below, we provide some examples to show that the prerequisite \eqref{X3} is valid.    Let
\begin{align*}
U(x)=x^2g_n(x)^2+a^2-2axg_n(x),\quad x\in\R   
\end{align*}
for some parameters $a>0$ and $n\ge1$, where 
\begin{equation*}
g_n(x)=
 \begin{cases}
   -n,\quad x\in(-\8,-n)\\
   x,\quad~~~ x\in[-n,n)\\
  n,\quad~~~ x\in(n,\8).
 \end{cases}   
\end{equation*}
Then, as shown in \cite{majka2020nonasymptotic}, 
$\hat b_t(x)=-\alpha_t U'(x),  $ along with $\hat \si_t(x)\equiv0$, 
satisfies \eqref{X3}, where  $\R\ni t\mapsto \alpha_t\in[0,\8)$ is a $\tau$-periodic function. In addition, for the $\tau$-periodic function $\alpha_\cdot$ above, we set 
$\hat b_t(x)=-\alpha_t\nn_x U(x)$ with $U(x)=\frac{1}{4}|x|^4-\frac{1}{2}|x|^2, x\in\R^d$. Then, by the aid of   \cite[Example 5.3]{cattiaux2014semi}, 
$\hat b_t$   satisfies \eqref{X3} in case of  $\hat \si_t(x)\equiv0$ for the high-dimensional setting. 
\color{black}

\begin{thm}\label{thm*}
{Assume  ({\bf H}) with $\int_0^\tau\alpha_u\,\d u\in(0,\8)$.} Then,  there exists a constant $K_2^*>0$ such that, for all $K_2\in(0,K_2^*]$, the stochastic semi-flow $\phi$, associated with \eqref{*X} and defined as in \eqref{E22}, has a unique  random $\tau$-periodic solution in the sense of distribution. Namely, there exists a unique (in the sense of distribution) stochastic process  $X^*_t\in L^1(\OO\to\R^d,\F_t,\P)$  such that for all $(s,t)\in\triangle$ and $ h\ge0,$ 
\begin{equation*}
X^*_{t+h}(\omega)\overset{d}{=}\phi(t+h, t, X^*_t(\omega),\omega),\qquad X^*_{t+\tau}(\omega)\overset{d}{=}X^*_t(\theta_\tau\omega).
\end{equation*}
Moreover, for all $t\in\R$ and $\xi\in L^1(\OO\to\R^d,\F_s,\P)$,
$$\lim_{s\downarrow-\8}X^{\xi}_{s,t}\overset{d}{=}X^*_t.$$
\end{thm}

Before proceeding, we comment on the random periodic solutions and time-periodic laws for McKean-Vlasov SDEs. 
\begin{remark}
{
As is known, the distribution flow $(\mu_t)_{t\ge0}$ of a McKean-Vlasov SDE solves a non-linear Fokker-Planck equation. In terminology, the deterministic  $(\mu_t)_{t\ge0}$ is said to be $\tau$-periodic if and only if $\mu_{i\tau}=\mu_0$ for any $i\ge1$ and $\mu_s\neq \mu_0$ for all $0\le s<\tau$; see, for instance,  \cite{Marini2023etalfrustratednetworks}.
The existence of time-periodic laws for McKean-Vlasov SDEs arising from the stochastic Brusselator was investigated in \cite{scheutzow1986periodic}, where the main tool used was the Schauder-Tychonoff fixed point theorem. 
To address the issue of time-periodic laws for (16) based on the strategy of \cite{scheutzow1986periodic},  additional non-trivial work would be required: i.e., choice of a weakly compact and convex sub-space on the Wasserstein space, construction of an appropriate mapping, and examination of weak continuity of the underlying mapping. 

Lastly, the notions and analyses for time-periodic laws and random periodic solutions are entirely different although the objects involved are related (but very distinct). 
The subject involved in time-periodic laws is deterministic, whereas the counterpart we discuss concerned with random periodic solutions is a stochastic process. 
}
\end{remark}

Concerning the McKean-Vlasov SDE \eqref{*X}, 
the corresponding  stochastic non-interacting particle system  and stochastic IPS are given respectively  as below: for $(s,t)\in\Delta$, 
\begin{equation}\label{P1-}
\d X_{s,t}^{i}=\big(\,\hat{b}_t(X_{s,t}^{i})+(\tilde{b}_t*\mu_{s,t}^i)(X_{s,t}^{i})\big)\,\d t+\sqrt{\alpha_t} \d B_t^i+\hat \si_t(X_{s,t}^{i})\,\d W_t^i
,  
\end{equation}
and 
\begin{equation}\label{P2-}
\d X_{s,t}^{i,N}
=
\bigg(\hat{b}_t(X_{s,t}^{i,N})+\frac{1}{N}\sum_{j=1}^N\tilde{b}_t(X_{s,t}^{i,N},X_{s,t}^{j,N})
\bigg)\,\d t+ \sqrt{\alpha_t} \d B_t^i
+\hat \si_t(X_{s,t}^{i,N})\,\d W_t^i
,
\end{equation}
where   $(X_{s,s}^i,B_t^{i},W_t^{i})_{i\in \mathbb S_N}$  are i.i.d. copies of $(X_{s,s},B_t,W_t)$,  $\mu_{s,t}^i:=\mathscr L_{X_{s,t}^{i}}$,  $(X_{s,s}^i,X_{s,s}^{i,N})_{i\in\mathbb S_N}$  are i.i.d. random variables.

\begin{thm}\label{thm4}
Assume  ({\bf H}) with $\int_0^\tau\alpha_u\,\d u\in(0,\8)$. Then, there exists a constant $K_2^*>0$ such that, for all $K_2\in(0,K_2^*]$ and  $N\ge1$,
the stochastic semi-flow $\phi^N$  corresponding to \eqref{P2-} and defined as in  \eqref{1T}, admits  a unique (in the sense of distribution) random $\tau$-periodic solution. That is, 
there exists a unique (in the sense of distribution) stochastic process $ {\bf X}^{*,N,N}_t =\text{vec}\big(X_{t}^{*,1,N},\cdots,X_{t}^{*,N,N}\big)\in L^2(\Omega\to(\R^{d})^N,\F_t,\P)$ such that  for all $t\in\R$ and $h\ge0$,
 \begin{equation*} 
{\bf X}^{*,N,N}_{t+h}(\omega)\overset{d}{=}\phi(t+h, t, {\bf X}^{*,N,N}_t(\omega),\omega),\qquad {\bf X}^{*,N,N}_{t+\tau}(\omega)\overset{d}{=}{\bf X}^{*,N,N}_t(\theta_\tau\omega).
\end{equation*}
Moreover,  for all $ t \in\R$ and $ \xi^N \in L^2(\Omega\to(\R^{d})^N,\F_s,\P)$,
\begin{equation*} 
    \lim_{s\downarrow-\8}{\bf X}^{N,N,\xi^N}_{s,t}\overset{d}{=}{\bf X}^{*,N,N}_t.
\end{equation*}
\end{thm}
We close this section with the final result of the particle limit $N\to \infty$ in the sense of distribution.
\begin{thm}\label{thm6}
Let  assumptions of Theorem \ref{thm4} hold. Then, for any $t\in\R$ and $i\ge1$,
\begin{equation*}
\lim_{N\to\8}{ X}^{*,i,N}_t\overset{d}{=}X_t^{*,i}.
\end{equation*}

\end{thm}

The proofs from Theorem \ref{thm*} to Theorem \ref{thm6} are postponed to Section \ref{sec:ProofofInteractingParticSystem:PoC:thM2}.


\section{Proof of Theorem \ref{th1}}
\label{sec:ProofofMVSDEwellposednesstheo:th1}
Before proceeding to complete the proof of Theorem \ref{th1},   we state and prove several preliminary lemmas. The first lemma we present shows that the solution process to \eqref{E1} has a uniformly bounded mean-square moment over $\Delta$. 

\begin{lemma}[Square-moment estimate] \label{lem3}
Under assumptions of Theorem \ref{th1}, there exist constants $\lambda, C>0$ and $\vv_0\in[0,1)$ such that for all $(s,t)\in\Delta$ and $\xi\in L^{2+\vv_0}(\OO\to\R^d,\F_s,\P)$, it holds that 
\begin{equation}\label{WW}
\E|X_{s,t}^\xi|^{2+\vv_0}\le C\big(1+\e^{-\lambda(t-s)}\E|\xi|^{2+
\vv_0}\big).
\end{equation}
\end{lemma}
\begin{proof}
With the aid of \eqref{E23}, there exist constants $\vv^*\in[0,1)$ and $\vv^{**}\in(0,1)$ such that 
\begin{equation}\label{W1}
-\lambda^*:=\int_0^\tau K^*(u) \,\d u\in(-\8,0) \quad \mbox{ with }\quad  K^{*}(u):=(1+\vv^*/2)\big(\hat{K}^{*}(u)+\tilde{K}^{^*}(u)\big),
\end{equation}
where 
\begin{equation*}
\begin{split}
\hat{K}^{*}(t):&=K_1(t)+\vv^{**}(1+K_3(t))+\frac{\big(1+K_2(t)+\vv^{**}K_3(t)+4K_3(t) +2(1+2K_3(t))\vv^*\big)\vv^*}{2+\vv^*},\\
\tilde{K}^{*}(t):&=\frac{2(K_2(t)+(\vv^{**}+2\vv^*) K_3(t))}{2+\vv^*}.
\end{split}
\end{equation*}
Below,  we shall stipulate $\vv^*\in(0,1)$
such that \eqref{W1} holds true ,  set $(s,t)\in\Delta$,  and fix  $\xi\in L^{2+\vv^*}(\OO\to\R^d,\F_s,\P)$. We omit the case $\vv^*=0$ as the proof is much simpler.
In accordance with \eqref{E20} and \eqref{E21}, we find that for all $t\in\R$, $x\in\R^d$ and $\mu\in \mathscr P_2(\R^d)$,
\begin{equation*}\label{W2}
\begin{split}
2\<x,b_t(x,\mu)\>+\|\sigma_t(x,\mu)\|_{\rm HS}^2&=2\<x,b_t(x,\mu)-b_t({\bf0},\delta_{\bf0})\>+\|\sigma_t(x,\mu)-\sigma_t({\bf0},\delta_{\bf0})\|_{\rm HS}^2\\
&\quad+2\<x,b_t({\bf0},\delta_{\bf0})\>+2\<\sigma_t(x,\mu)-\sigma_t({\bf0},\delta_{\bf0}),\sigma_t({\bf0},\delta_{\bf0})\>_{\rm HS}+\|\sigma_t({\bf0},\delta_{\bf0})\|_{\rm HS}^2\\
&\le \big(K_1(t)+\vv^{**}(1+K_3(t))\big)|x|^2+\big(K_2(t)+\vv^{**} K_3(t)\big)\mathbb W_2(\mu,\delta_{\bf0})^2\\
&\quad+ |b_t({\bf0},\delta_{\bf0})|^2/\vv^{**}+(1+1/\vv^{**})\|\si_t({\bf0},\delta_{\bf0})\|_{\rm HS}^2,
\end{split}
\end{equation*}
where, for the third and fourth terms on the RHS of the first equality, we make use the fact that $|ab|\le\vv^{**}a^2+b^2/\vv^{**}$.
Then, we apply Young's inequality, i.e., $ab\le a^p/p+b^q/q$ with $a,b>0$ and $p=\frac{2+\vv^{*}}{\vv^{*}}$ and $q=\frac{2+\vv^*}{2}$, to the following estimate:  for all  $t\in\R$, $x\in\R^d$ and $\mu\in \mathscr P_2(\R^d)$,
\begin{equation} \label{T}
\begin{split}
|x|^{\vv^*}\big(2\<x,b_t(x,\mu)\>+\|\sigma_t(x,\mu)\|_{\rm HS}^2\big)&\le\big(  K_1(t)+\vv^{**}(1+K_3(t))  \big)|x|^{2+\vv^*}\\
&\quad+|x|^{\vv^*}\big(K_2(t)+\vv^{**} K_3(t)\big)\mathbb W_2(\mu,\delta_{\bf0})^2\\
&\quad+\Big( |b_t({\bf0},\delta_{\bf0})|^2/\vv^{**}+(1+1/\vv^{**})\|\si_t({\bf0},\delta_{\bf0})\|_{\rm HS}^2\Big)|x|^{\vv^*} \\
&\le\bigg(  K_1(t)+\vv^{**}(1+K_3(t))  +\frac{(1+K_2(t)+\vv^{**} K_3(t))\vv^*}{2+\vv^*}\bigg)|x|^{2+\vv^*}\\
&\quad+\frac{2(K_2(t)+\vv^{**} K_3(t)) }{2+\vv^*}\mathbb W_2(\mu,\delta_{\bf0})^{2+\vv^*}
\\
&\quad+\frac{2}{2+\vv^*} \Big(|b_t({\bf0},\delta_{\bf0})|^2/\vv^{**}+(1+1/\vv^{**})\|\si_t({\bf0},\delta_{\bf0})\|_{\rm HS}^2\Big)^{1+\frac{\vv^*}{2}}.
\end{split}
\end{equation}
 Moreover, because of \eqref{E21}, using the same trick from Young's inequality above, we infer that for all $t\in\R$, $x\in\R^d$ and $\mu\in \mathscr P_2(\R^d)$,
\begin{equation} \label{T*alter}
\begin{split}
|x|^{\vv^*}\|\si_t(x,\mu)\|_{\rm HS}^2&\le 2 |x|^{\vv^*}\big(\|\si_t(x,\mu)-\sigma_t({\bf0},\delta_{{\bf0}})\|_{\rm HS}^2+\| \sigma_t({\bf0},\delta_{{\bf0}})\|_{\rm HS}^2\big)\\
&\le 2 |x|^{\vv^*}\big(K_3(t)|x|^2+K_3(t)\mathbb W_2(\mu,\delta_{\bf 0})^2+\| \sigma_t({\bf0},\delta_{{\bf0}})\|_{\rm HS}^2\big)\\
&\le  \frac{4K_3(t)+2(1+2K_3(t))\vv^*}{2+\vv^*}|x|^{2+\vv^*}\\
&\quad+\frac{4K_3(t) }{2+\vv^*}\mathbb W_2(\mu,\delta_{\bf 0})^{2+\vv^*}+\frac{4}{2+\vv^*}\| \sigma_t({\bf0},\delta_{{\bf0}})\|_{\rm HS}^{2+\vv^*}.
\end{split}
\end{equation}
Next, applying It\^o's formula and taking expectations yields 
\begin{equation*}
\begin{split}
\e^{-\int_s^tK^*(u)\,\d u}\E|X_{s,t}^\xi|^{2+\vv^*}&\le \E|\xi|^{2+\vv^*}+\int_s^t\e^{-\int_s^uK^*(v)\,\d v}\big(-K^*(u)\E|X_{s,u}^\xi|^{2+\vv^*}\\
&\quad+\mathbb{E}\big[(1+\vv^*/2)|X_{s,u}^\xi|^{\vv^*}\big(2\<X_{s,u}^\xi,b_u(X_{s,u}^\xi,\mathscr L_{X_{s,u}^\xi})\>\\
&\quad+\|\si_u(X_{s,u}^\xi,\mathscr L_{X_{s,u}^\xi})\|_{\rm HS}^2\big)\big]\\
&\quad+(1+\vv^*/2)\vv^*\mathbb{E}\big[|X_{s,u}^\xi|^{\vv^*}\|\si_u(X_{s,u}^\xi,\mathscr L_{X_{s,u}^\xi})\|_{\rm HS}^2\big]\big)\,\d u.
\end{split}
\end{equation*}
Next, substituting \eqref{T} and \eqref{T*alter} into the inequality above yields 
\begin{equation*}
\begin{split}
&\e^{-\int_s^tK^*(u)\,\d u}\E|X_{s,t}^\xi|^{2+\vv^*}
\\
&\le
\E|\xi|^{2+\vv^*}+(1+\vv^*/2)\int_s^t\e^{-\int_s^u K^\vv(v)\,\d v} \tilde{K}^*(u) \big(-\E|X_{s,u}^\xi|^{2+\vv^*}+\mathbb W_2\big(\mathscr L_{X_{s,u}^\xi},\delta_{{\bf0}}\big)^{2+\vv^*}\big)\,\d u
\\
&\quad+C_0(\vv^*,\vv^{**})\int_s^t\e^{-\int_s^u K^\vv(v)\,\d v}\big(|b_u({\bf0},\delta_{\bf0})|^2+\|\si_u({\bf0},\delta_{\bf0})\|_{\rm HS}^2\big)^{1+\frac{\vv^*}{2}}\d u ,
\end{split}
\end{equation*}
for some constant $C_0(\vv^*,\vv^{**})>0.$
Subsequently, combining with the fact that for any $\mu\in\mathscr P_{2+\vv^*}(\R^d)$,
$$\mathbb W_2(\mu,\delta_{{\bf0}})^{2+\vv^*}\le \mathbb W_{2+\vv^*}(\mu,\delta_{{\bf0}})^{2+\vv^*}=\mu(|\cdot|^{2+\vv^*}),$$
we obtain that 
\begin{equation*}
 \E|X_{s,t}^\xi|^{2+\vv^*}\le\e^{\int_s^tK^\vv(u)\,\d u}\E|\xi|^{2+\vv^*} +C_0(\vv^*,\vv^{**})\int_s^t\e^{ \int_u^tK^*(v)\,\d v}\big(|b_u({\bf0},\delta_{\bf0})|^2+\|\si_u({\bf0},\delta_{\bf0})\|_{\rm HS}^2\big)\d u.
\end{equation*}
Finally, by making use  of Lemma  \cite[Lemma 3.1]{bao2022random} and taking the continuity and the periodicity of $K_1,K_2,K_3$ into consideration, there is a constant $C_1(\vv^*,\vv^{**})>0$ (independent of $\xi$ and time parameters $t,s$) such that 
\begin{equation*}
\begin{split}
 \E|X_{s,t}^\xi|^{2+\vv^*}&\le C_1(\vv^*,\vv^{**})\e^{-\lambda^*\lfloor (t-s)/\tau\rfloor}\E|\xi|^{2+\vv^*} +C_1(\vv^*,\vv^{**})\int_s^t\e^{-\lambda^*\lfloor (t-u)/\tau\rfloor} \d u. 
\end{split}
\end{equation*}
Whence, the assertion \eqref{WW} immediately follows. 
\end{proof}

The following lemma reveals that the solution processes are continuous in the mean-square sense w.r.t. the initial values uniformly w.r.t. the time parameters.

\begin{lemma}[Moments of differences]\label{lem:stable1}
Under  assumptions of Theorem \ref{th1}, there exists a constant $C>0$  such that for all $(s,t)\in\Delta$  and $\xi,\eta\in L^2(\OO\to\R^d,\F_s,\P)$,
\begin{equation}\label{Wu}
   \mathbb{E} |X_{s,t}^{\xi}-Y_{s,t}^{\eta}|^2 \le C\,\e^{-\frac{\lambda}{\tau}  (t-s)  }\E|\xi-\eta|^2 
   \quad\textrm{and}\quad \mathbb{W}_2\big( \mathscr L_{X_{s,t}^\xi},\mathscr L_{X_{s,t}^\eta}\big)^2
   \leq 
   C\,\e^{-\frac{\lambda}{\tau}  (t-s)  }
   \mathbb{W}_2\big( \mathscr L_{\xi},\mathscr L_{\eta}\big)^2,
\end{equation}
 where $\lambda>0$ was defined in \eqref{E23}. 
\end{lemma}

\begin{proof}
In the sequel, we still write $(s,t)\in\Delta$ and $\xi\in L^2(\OO\to\R^d,\F_s,\P)$ and moreover set 
$\psi_{s,t}:=X_{s,t}^{\xi}-Y_{s,t}^{\eta}$ for the sake  of  notation brevity.
For $\bar K(t):=K_1(t)+K_2(t)$, $t\in\R,$ applying It\^o's formula, we deduce from \eqref{E20} that 
\begin{equation*}
\begin{split}    
\e^{-\int_s^t\bar K(u)\,\d u}\E|\psi_{s,t}|^2&=\E|\xi-\eta|^2+\int_s^t\e^{-\int_s^u\bar K(v)\,\d v}\big(-\bar K(u)\E|\psi_{s,u}|^2\\
&\qquad\qquad\qquad\qquad\qquad\qquad\quad+2\E\<\psi_{s,u}, b_u(X_{s,u}^\xi,\mathscr L_{X_{s,u}^\xi})-b_u(X_{s,u}^\eta,\mathscr L_{X_{s,u}^\eta})\>\\
&\qquad\qquad\qquad\qquad\qquad\qquad\quad+\E\|\sigma_u(X_{s,u}^\xi,\mathscr L_{X_{s,u}}^\xi)-\si_u(X_{s,u}^\eta,\mathscr L_{X_{s,u}}^\eta)\|_{\rm HS}^2\big)\,\d u\\
&\le \E|\xi-\eta|^2+\int_s^t\e^{-\int_s^u\bar K(v)\,\d v}K_2(u)\big(-  \E|\psi_{s,u}|^2+\mathbb W_2\big(\mathscr L_{X_{s,u}^\xi},\mathscr L_{X_{s,u}^\eta}\big)^2\big)\,\d u\\
&\le \E|\xi-\eta|^2,
\end{split}
\end{equation*}
where in the last display we used the basic fact that 
\begin{equation}\label{ETT}
   \mathbb W_2\big(\mathscr L_{X_{s,u}^\xi},\mathscr L_{X_{s,u}^\eta}\big)^2\le  \E|X_{s,u}^{\xi}-Y_{s,u}^{\eta}|^2=\E|\psi_{s,u}|^2,\qquad u\ge s.
\end{equation}
Therefore, we arrive at 
\begin{equation*}
\begin{split} 
\E|\psi_{s,t}|^2\le  \e^{\int_s^t\bar K(u)
\,\d u}\E|\xi-\eta|^2.
 \end{split}
\end{equation*}
 Then, invoking \cite[Lemma 3.1]{bao2022random} and recalling the definition of $\lambda$ introduced in \eqref{E23}  yields 
 \begin{equation*} 
   \mathbb{E} |\psi_{s,t}|^2 \le \exp\bigg(-\lambda\lfloor (t-s)/\tau\rfloor +\int_{s-\lfloor s/\tau\rfloor \tau}^{t-(\lfloor s/\tau\rfloor+\lfloor (t-s)/\tau\rfloor) \tau}\bar K(u) \,\d u\bigg)\E|\xi-\eta|^2.
\end{equation*}
 This, together with the local  integrability of $\bar K$ 
 due to the corresponding continuity of $K_1$ and $K_2$, leads to the desired assertion, i.e., the first inequality in \eqref{Wu}.
 The second inequality in \eqref{Wu} is followed by choosing $\xi,\eta\in L^2(\OO\to\R^d,\F_s,\P)$ satisfying $\E|\xi-\eta|^2=\mathbb{W}_2\big( \mathscr L_{\xi},\mathscr L_{\eta}\big)^2$ on the RHS of the first inequality in \eqref{Wu} while using \eqref{ETT} for it on the LHS.
\end{proof}

The subsequent lemma is concerned with the stochastic semi-flow property of the solution process, which is the foundation to treat   existence of random periodic solutions. 

\begin{lemma}\label{lem1}
Under Assumption $({\bf A})$, the solution process $(X_{s,t}^\xi)_{t\ge s}$ admits the semi-flow property:
\begin{equation}\label{E5}
X_{s,t}^\xi=X_{r,t}^{X_{s,r}^\xi},\quad s\le r\le t,\quad \xi\in H:=L^2(\OO\to\R^d,\F_s,\P).
\end{equation}
\end{lemma}

\begin{proof} 
For any $T\ge s $ and $\nu\in \mathscr P_2(\R^d)$, let 
\begin{equation*}
\mathcal C_{s,T}^\nu=\big\{\mu\in C\big([s,T];\mathscr P_2(\R^d)\big):\mu_{s}=\nu\big\}.
\end{equation*}
  Define a metric   on $\mathcal C_{s,T}^\nu$ as below: for any $\ll>0,$
\begin{equation*}
\mathbb W_{2,\ll}\big(\mu^1,\mu^2\big):=\sup_{s\le t\le T}\big(\e^{-\ll (t-s)}\mathbb W_2\big(\mu^1_{t},\mu^2_{t}\big)\big),\quad \mu^1,\mu^2\in\mathcal C_{s,T}^\nu.
\end{equation*}
Under the metric $\mathbb W_{2,\ll}$, $\mathcal C_{s,T}^\nu$ is a Polish space. For a fixed $\mu\in\mathcal C_{s,T}^\nu$, we consider the following distribution-frozen SDE associated with \eqref{E1}
\begin{equation}\label{E4}
\d Y_{s,t}^{\mu,\xi}=b_t(Y_{s,t}^{\mu,\xi},\mu_{t})\,\d t+\si_t(Y_{s,t}^{\mu,\xi},\mu_{t})\,\d W_t,\qquad  (s,t)\in\Delta,\qquad Y_{s,s}^{\mu,\xi}=X_{s,s}^{\xi}=\xi\sim\nu,
\end{equation}
where $X_{s,s}^\xi$ is the initial value of the McKean-Vlasov SDE \eqref{E1}. 
Under Assumption ({\bf A}), the classical time-inhomogeneous  SDE \eqref{E4} is strongly well-posed (see e.g. \cite[Theorem 3.1.1, p.44]{prevot2007concise}) so the associated solution process $(Y_{s,t}^{\mu,\xi})_{t\ge s}$
possesses the semi-flow property (in $\bR^d$ and hence in $L^2$), i.e., 
\begin{equation}\label{E7}
Y_{s,t}^{\mu,\xi}=Y_{r,t}^{\mu,Y_{s,r}^{\mu,\xi}},\quad s\le r\le t, \quad \xi\in H. 
\end{equation}
For any $\mu\in\mathcal C_{s,T}^\nu$, define the mapping $\Phi(\mu)$ as below
\begin{equation}\label{E6}
\big(\Phi(\mu)\big)_{s,t,\xi}=\mathscr L_{Y_{s,t}^{\mu,\xi}},\quad s\le t\le T,\quad\xi\in H.
\end{equation}
Provided that the mapping $\Phi \in \mathcal C_{s,T}^\nu$ defined above admits a unique fixed point $\mu\in \mathcal C_{s,T}^\nu$, i.e., 
\begin{equation*}
\mu=\Phi(\mu)\qquad \mbox{ so } \qquad \big(\Phi(\mu)\big)_{s,t,\xi}=\mu_{t},\quad s\le t\le T.
\end{equation*}
Thus, in addition to \eqref{E6}, we can replace the measure variable $\mu_{t}$ in \eqref{E4} by $\mathscr L_{Y_{s,t}^{\mu,\xi}},$ so the SDE \eqref{E4} is exactly the same as the McKean-Vlasov SDE \eqref{E1}. Consequently, we find 
\begin{equation*}
X_{s,t}^{\xi}=Y_{s,t}^{\mu,\xi},\quad s\le t\le T,\quad \xi\in  H.
\end{equation*}
This, besides  \eqref{E7}, enables us to derive that 
\begin{equation*} 
X_{s,t}^{\xi}=X_{r,t}^{X_{s,r}^{\xi}},\quad s\le r\le t\le T, \quad \xi\in H.
\end{equation*}
Accordingly, the semi-flow property \eqref{E5} holds true.

Based on the analysis above, to achieve the semi-flow property \eqref{E5}, it remains to show that the mapping $\Phi$ defined in \eqref{E6} has a unique fixed point in $\mathcal C_{s,T}^\nu$. To this end, it suffices   to justify 
\begin{enumerate} 
\item[(i)] $\Phi$ maps $\mathcal C_{s,T}^\nu$ into itself, i.e., $\Phi:\mathcal C_{s,T}^\nu\rightarrow\mathcal C_{s,T}^\nu$;

\item[(ii)] $\Phi$  is contractive in $\mathcal C_{s,T}^\nu$ under the metric $\mathbb W_{2,\ll}$ for   some suitably chosen  $\ll>0$, i.e., there exists a constant $c\in(0,1)$ such that for some $\ll>0$ to be figured,
    \begin{equation*}
\mathbb W_{2,\ll}\big(\Phi(\mu^1),\Phi(\mu^2)\big)\le c\,\mathbb W_{2,\ll}\big(\mu^1,\mu^2\big).
    \end{equation*}
\end{enumerate}
In the first place, we verify the assertion (i). For any $t\in\R $,   $x\in\R^d$ and $\rho\in\mathscr P_2(\R^d)$, we infer from 
\eqref{E2} and \eqref{E3} that 
\begin{equation}\label{E8}
2\<x,b_t(x,\rho)\>\le \big(1+\bar{K}_1(t)\big)|x|^2+ \bar K_1(t) \mathbb W_2(\rho,\delta_{{\bf0}})^2+|b_t({\bf0},\delta_{{\bf0}})|^2,
\end{equation}
and
\begin{equation} \label{E8E}
\|\sigma_t(x,\rho)\|_{\rm HS}^2\le 2\bar{K}_2(t)\big(|x|^2+\mathbb W_2(\rho,\delta_{{\bf0}})^2\big)+2\|\sigma_t({\bf 0},\delta_{{\bf0}})\|_{\rm HS}^2.
\end{equation}

Next, 
applying It\^o's formula followed by the Burkholder-Davis-Gundy (BDG) inequality, we deduce that   for any   $\xi\in H$  and $s\le t\le T$,
\begin{equation*}
\begin{split}
\E\Big(\sup_{s\le r\le t}|Y_{s,r}^{\mu,\xi}|^2\Big)&\le \E|\xi|^2+\int_s^t\big\{2\<Y_{s,u}^{\mu,\xi},b_u(Y_{s,u}^{\mu,\xi},\mu_{u})\>+\|\si_u(Y_{s,u}^{\mu,\xi},\mu_{u})\|_{\rm HS}^2\big\}\,\d u\\
&\quad+2\E\bigg(\sup_{s\le r\le t}\int_s^r\<Y_{s,u}^{\mu,\xi},\si_u(Y_{s,u}^{\mu,\xi},\mu_{u})\d W_u\>\bigg)\\
&\le \E|\xi|^2+\int_s^t\big\{2\<Y_{s,u}^{\mu,\xi},b_u(Y_{s,u}^{\mu,\xi},\mu_{u})\>+65\|\si_u(Y_{s,u}^{\mu,\xi},\mu_{u})\|_{\rm HS}^2\big\}\,\d u+\frac{1}{2}\E\Big(\sup_{s\le r\le t}|Y_{s,r}^{\mu,\xi}|^2\Big).
\end{split}
\end{equation*}
Thereby,   \eqref{E8} and \eqref{E8E} yield 
\begin{equation*}
\begin{split}
\E\Big(\sup_{s\le r\le t}|Y_{s,r}^{\mu,\xi}|^2\Big)
&\le 2\E|\xi|^2+2\int_s^t\big\{2\<Y_{s,u}^{\mu,\xi},b_u(Y_{s,u}^{\mu,\xi},\mu_{u})\>+65\|\si_u(Y_{s,u}^{\mu,\xi},\mu_{u})\|_{\rm HS}^2\big\}\,\d u\\
&\le 2\E|\xi|^2+\int_s^t\ll_{u}\big(\E|Y_{s,u}^{\mu,\xi}|^2+\mathbb W_2(\mu_{u},\delta_{{\bf0}})^2+C_u\big)\,\d u,
\end{split}
\end{equation*}
where
$$\ll_u:=\bar{K}_1(u)+130\bar{K}_2(u)+1 \quad \mbox{ and } \quad  C_u:= |b_u({\bf0},\delta_{{\bf0}})|^2+130\|\si_u({\bf0},\delta_{{\bf0}})\|^2_{\rm HS},\qquad u\ge s.$$
Thus, applying Gronwall's inequality leads to
\begin{equation}\label{T*}
\begin{split}
\E\Big(\sup_{s\le r\le t}|Y_{s,r}^{\mu,\xi}|^2\Big)&\le \bigg(2\E|\xi|^2+\int_s^t\big\{\lambda_u\mathbb W_2(\mu_{s,u},\delta_{{\bf0}})^2+C_u\big\}\,\d u\bigg)\e^{\int_s^t\ll_u\,\d u}.
\end{split}
\end{equation}
This, combining $\mu \in\mathcal C_{s,T}^\nu$ with the fact that 
$t\mapsto \bar{K}_1(t)+\bar{K}_2(t)$ and $t\mapsto |b_t({\bf0},\delta_{{\bf0}})|^2+\|\si_t({\bf0},\delta_{{\bf0}})\|^2_{\rm HS}$ are locally integrable, yields $\mathscr L_{Y_{s,t}^{\mu,\xi}}\in \mathscr P_2(\R^d)$ for any $\mu \in\mathcal C_{s,T}^\nu$, $\xi\in H$  and $s\le t\le T$. In the following, to examine  $\Phi\in\mathcal C_{s,T}^\nu$,
we further need to claim the continuity, ie,
\begin{equation}\label{TT}
\lim_{\triangle t\downarrow 0}\mathbb W_2\big(\mathscr L_{Y_{s,t+\triangle t}^{\mu,\xi}},\mathscr L_{Y_{s,t}^{\mu,\xi}}\big)=0. 
\end{equation}
In the light of the semi-flow property of $\mathscr L_{Y_{s,t}^{\mu,\xi}}$, it is sufficient  to prove  \begin{equation} \label{WW*}
\lim_{\triangle s\downarrow0}\mathbb W_2\big(\mathscr L_{Y_{s,s+\triangle s}^{\mu,\xi}},\mathscr L_\xi\big)=0. 
\end{equation} 
To end this,  define the stopping time for each integer  $n\ge1$,
$$\tau_n:=\inf\big\{t\ge s: |Y_{s,t}^{\mu,\xi}|\ge n\big\}.$$
In the following analysis, we shall  stipulate $\triangle s\in(0,1)$. 
Apparently, we have 
\begin{equation*}
\begin{split}
\E|Y_{s,s+\triangle s}^{\mu,\xi}-\xi|^2&=\E\big(|Y_{s,s+\triangle s}^{\mu,\xi}-\xi|^2{\bf1}_{\{s+\triangle s<\tau_n\}}\big)+\E\big(|Y_{s,s+\triangle s}^{\mu,\xi}-\xi|^2{\bf1}_{\{{s+}\triangle s\ge\tau_n\}}\big)\\ &=:I_1(n,\triangle s)+I_2(n,\triangle s).
\end{split}
\end{equation*}
Since, for each $t\in\R,$  $b_t$ and $\si_t$
are  continuous and bounded on bounded sets of $\R^d\times \mathscr P_2(\R^d)$,  we have for each fixed $n\ge1$,
\begin{equation}\label{W*}
\lim_{\triangle s\to0}I_1(n,\triangle s)=0 \quad \bP\textrm{-a.s.}
\end{equation}
On the other hand, it is easy to see that
\begin{equation*}
\begin{split}
I_2(n,\triangle s)\le 2\big(\E\big(|Y_{s,s+\triangle s}^{\mu,\xi}|^2+|\xi|^2\big){\bf1}_{\{{s+}\triangle s\ge\tau_n\}}\big).
\end{split}
\end{equation*}
By the apparent fact that 
\begin{equation*}
\{{s+}\triangle s\ge \tau_n\}\subseteq\Big\{\sup_{s\le r\le s+1}|Y_{s,s+r}^{\mu,\xi}|\ge n\Big\},
\end{equation*}
we are ready to see that  
\begin{equation*}
\begin{split}
I_2(n,\triangle s)\le 4 \E\Big(\sup_{s\le r\le s+1}|Y_{s,r}^{\mu,\xi}|^2 {\bf1}_{\big\{\sup_{s\le r\le s+1}|Y_{s,r}^{\mu,\xi}|\ge n\big\}}\Big).
\end{split}
\end{equation*}
Consequently, taking \eqref{T*} into consideration and  making use of the uniform integrability of the sequence $\sup_{s\le r\le s+1}|Y_{s,s+r}^{\mu,\xi}|$, we arrive at 
$$\lim_{n\to\8}I_2(n,\triangle s)=0.$$
Therefore, combining this with \eqref{W*}, we reach \eqref{WW*}.

In the later context, we aim to demonstrate  the assertion (ii). Below, we shall assume that  $s\le t\le T$, $\mu^1,\mu^2\in \mathcal C_{s,T}^\nu$ and $\xi\in H$.
Again, by invoking It\^o's formula, we deduce from  \eqref{E2} and \eqref{E3} that 
\begin{equation*}
\E|Y_{s,t}^{\mu_1,\xi}-Y_{s,t}^{\mu_2,\xi}|^2\le  \int_s^t\big(\bar{K}_1(u)+\bar{K}_2(u)\big)\big(\E|Y_{s,u}^{\mu_1,\xi}-Y_{s,u}^{\mu_2,\xi}|^2+\mathbb W_2\big(\mu^1_{u},\mu_{u}^2\big)^2\big)\d u.
\end{equation*}
Subsequently,  Gronwall's inequality yields  
\begin{equation*}
\E|Y_{s,t}^{\mu_1,\xi}-Y_{s,t}^{\mu_2,\xi}|^2\le  \int_s^t\big(\bar{K}_1(u)+\bar{K}_2(u)\big)\mathbb W_2\big(\mu^1_u,\mu_u^2\big)^2\d u\exp\bigg(\int_s^t\big(\bar{K}_1(u)+\bar{K}_2(u)\big)\,\d u
\bigg).
\end{equation*}
This further implies that for any $\ll>0$,
\begin{equation}\label{E9}
\begin{split}
\e^{-2\ll (t-s)}\E|Y_{s,t}^{\mu_1,\xi}-Y_{s,t}^{\mu_2,\xi}|^2&\le   \int_s^t\big(\bar{K}_1(u)+\bar{K}_2(u)\big)\e^{-2\ll (t-u)}\e^{-2\ll (u-s)}\mathbb W_2(\mu^1_u,\mu_u^2)^2\d u\\
&\quad\times\exp\bigg(\int_s^T\big(\bar{K}_1(u)+\bar{K}_2(u)\big)\,\d u\bigg)\\
&\le \sup_{s\le t\le T}\int_s^t\big(\bar{K}_1(u)+\bar{K}_2(u)\big)\e^{-2\ll (t-u)} \d u\\
&\quad\times\exp\bigg(\int_s^T\big(\bar{K}_1(u)+\bar{K}_2(u)\big)\,\d u\bigg)\mathbb W_{2,\ll}\big(\mu^1,\mu^2\big)^2.
\end{split}
\end{equation}
Observe  that
\begin{equation*}
\begin{split}
\mathbb W_{2,\ll}(\Phi(\mu^1 ),\Phi(\mu^2 ))&=\sup_{s\le t\le T}\big(\e^{-\ll (t-s)}\mathbb W_2\big((\Phi(\mu^1 ))_{s,t,\xi},(\Phi(\mu^2 ))_{s,t,\xi}\big)\big)\\
&\le \sup_{s\le t\le T}\big(\e^{-2\ll (t-s)}\E|Y_{s,t}^{\mu_1,\xi}-Y_{s,t}^{\mu_2,\xi}|^2\big)^{1/2}.
\end{split}
\end{equation*}
Thus, we conclude from \eqref{E9} that the assertion (ii) is valid by using the fact that both $\bar{K}_1(\cdot)$ and $\bar{K}_2(\cdot)$ are locally integrable followed by approaching $\ll>0$ sufficiently large.
\end{proof}

\begin{lemma}[Periodicity]\label{lem2}
Let Assumption ({\bf A}) hold  and suppose further that, for all fixed $x\in\R^d$ and $\mu\in\mathscr P(\R^d)$, $t\mapsto b_t(x,\mu)$ and $t\mapsto \sigma_t(x,\mu)$ are $\tau$-periodic for some constant $\tau>0$. 
Then, for any $(s,t)\in\Delta$, $\omega\in\Omega$ and $\xi\in L^2(\OO\to\R^d,\F_{s},\P)$, 
\begin{equation}\label{W16}
X_{s+\tau,t+\tau}^\xi(\omega)=X_{s,t}^\xi(\theta_\tau\omega).
\end{equation}
\end{lemma}

\begin{proof}
In the following analysis, we prescribe  $s\le t\le T$, $\omega\in\OO$ and $X_{s,s}^\xi=\xi\in L^2(\OO\to\R^d,\F_s,\P)$ with $\nu:=\mathscr L_\xi$. For our purpose, we set 
\begin{equation*}
\tilde{\mathcal C}_{s,T+\tau}^\nu:=\big\{\mu\in C([s,T+\tau]\times[s,T+\tau];\mathscr P_2(\R^d)):\mu_{s,s}=\nu,~\mu_{r,\tilde{r}}=\mu_{r+\tau,\tilde{r}+\tau},~ s\le r\le  \tilde{r}\le T\big\}.
\end{equation*}
Below, we still work on the SDE \eqref{E4} but with  $\mu\in\tilde{\mathcal C}_{s,T+\tau}^\nu$ and write   $(Y_{s,t}^{\mu,\xi})_{t\ge s}$ as  the solution to \eqref{E4} with $\mu\in\tilde{\mathcal C}_{s,T+\tau}^\nu$. In the light of \eqref{E4}, we deduce that 
\begin{equation}\label{E10}
\begin{split}
Y_{s+\tau,t+\tau}^{\mu,\xi}(\omega)&=\xi+\int_{s+\tau}^{t+\tau}b_u\big(Y_{s+\tau,u}^{\mu,\xi}(\omega),\mu_{s+\tau,u}\big)\,\d u+\int_{s+\tau}^{t+\tau}\si_u\big(Y_{s+\tau,u}^{\mu,\xi}(\omega),\mu_{s+\tau,u}\big)\,\d \omega(u)\\
&=\xi+\int_s^tb_{u+\tau}\big(Y_{s+\tau,u+\tau}^{\mu,\xi}(\omega),\mu_{s+\tau,u+\tau}\big)\,\d u\\
&\quad+\int_s^t\si_{u+\tau}\big(Y_{s+\tau,u+\tau}^{\mu,\xi}(\omega),\mu_{s+\tau,u+\tau}\big)\,\d (\theta_\tau\omega)(u)\\
&=\xi+\int_s^tb_{u}\big(Y_{s+\tau,u+\tau}^{\mu,\xi}(\omega),\mu_{s,u}\big)\,\d u +\int_s^t\si_u\big(Y_{s+\tau,u+\tau}^{\mu,\xi}(\omega),\mu_{s,u}\big)\,\d (\theta_\tau\omega)(u),
\end{split}
\end{equation}
where the second identity is due to the variable substitution strategy and the third identity holds true since $b$ and $\sigma$ are $\tau$-periodic with respect to the time variable and $\mu_{s+\tau,u+\tau}=\mu_{s,u}$ for all $s\le u\le t$ owing to $\mu\in\tilde{\mathcal C}_{s,T+\tau}^\nu$.  Furthermore, we infer readily from \eqref{E4} that 
\begin{equation}\label{E11}
Y_{s,t}^{\mu,\xi}(\theta_\tau\omega)=\xi+\int_{s}^{t}b_u\big(Y_{s,u}^{\mu,\xi}(\theta_\tau\omega),\mu_{s,u}\big)\,\d u+\int_s^t\si_u\big(Y_{s,u}^{\mu,\xi}(\theta_\tau\omega),\mu_{s,u}\big)\,\d (\theta_\tau\omega)(u).
\end{equation}
Now, combining \eqref{E10}  with \eqref{E11} and taking   the strong wellposedness of \eqref{E4} into account gives that 
 for all $\mu\in\tilde{\mathcal C}_{s,T+\tau}^\nu$,
\begin{equation}\label{W12}
Y_{s+\tau,t+\tau}^{\mu,\xi}(\omega)=Y_{s,t}^{\mu,\xi}(\theta_\tau\omega).
\end{equation}
For $\mu\in\tilde{\mathcal C}_{s,T+\tau}^\nu$, we define the map $\mu\mapsto \Phi(\mu)$ as 
\begin{equation*}
(\Phi(\mu))_{s,t,\xi}=\mathscr L_{Y_{s,t}^{\mu,\xi}}.
\end{equation*}
Once we can show that $\Phi:\tilde{\mathcal C}_{s,T+\tau}^\nu\to \tilde{\mathcal C}_{s,T+\tau}^\nu$ has a unique fixed point, written as $\mu\in \tilde{\mathcal C}_{s,T+\tau}^\nu$, we then obviously have for all $\mu\in\tilde{\mathcal C}_{s,T+\tau}^\nu $ , 
\begin{equation}\label{W17}
(\Phi(\mu))_{s,t,\xi}=\mu_{s,t}=\mathscr L_{Y_{s,t}^{\mu,\xi}}.
\end{equation}
This enables us to obtain that $Y_{s,t}^{\mu,\xi}=X_{s,t}^\xi$. As a result, \eqref{W16} follows from \eqref{W12}.  Consequently, to obtain \eqref{W16}, it is sufficient to show that  $\Phi:\tilde{\mathcal C}_{s,T+\tau}^\nu\to \tilde{\mathcal C}_{s,T+\tau}^\nu$ has a unique fixed point. 
It is trivial to see that  $(\Phi(\mu))_{s,s,\xi}=\mathscr L_{Y_{s,s}^{\mu,\xi}}=\mathscr L_\xi=\nu$. On the other hand, making use  of \eqref{W12} and
taking the measure-preserving property of $\P$, we  have 
\begin{equation*}
(\Phi(\mu))_{s+\tau,t+\tau,\xi}=(\Phi(\mu))_{s,t,\xi}.
\end{equation*}
Next, by following exactly the argument of Lemma \ref{lem1}, we conclude that the map $\Phi:\tilde{\mathcal C}_{s,T+\tau}^\nu\to \tilde{\mathcal C}_{s,T+\tau}^\nu$ possesses a unique fixed point in $\tilde{\mathcal C}_{s,T+\tau}^\nu$.  
\end{proof}

With the preparations above at hand, we are in a position to complete the proof of Theorem \ref{th1}.
\begin{proof}[Proof of Theorem \ref{th1}]
Let $\phi$ be the mapping defined in \eqref{E22}.
In accordance to \cite[Proposition 3.6]{bao2022random}, to finish the proof of Theorem \ref{th1},
we need to check that   for  all $(s,t)\in\Delta$ and $\xi,\eta\in L^2(\OO\to\R^d,\F_s,\P)$,
\begin{itemize}
\item[(a)] {\bf semi-flow property}:  $\phi(s,t,\xi,\omega)=\phi(r,t,\phi(s,r,\xi,\omega),\omega)$ \qquad $\bP$-a.s.;

\item[(b)]{\bf time-shift equals omega-shift}:
$\phi(t+\tau,s+\tau,\xi,\omega)=\phi(t,s,\xi,\theta_\tau\omega)$\qquad $\bP$-a.s.;

\item[(c)] {\bf contractive property}: there exists a decreasing function $h:\R\to\R_+$  such that
$$\E|\phi(t,s,\xi)-\phi(t,s,\eta)|^2\le h(t-s)\E|\xi-\eta|^2,$$
where for some $\tau_0>0$,
\begin{equation*} 
    \lim_{s\to-\8}\sum_{j=0}^\8h(t-s+j\tau_0)=0;
\end{equation*}

\item[(d)] {\bf Ultimate boundedness  in the mean-square sense}:   there is a constant $C>0$ (independent of $\xi$) such that
$$\sup_{t\ge s}\E|\phi(t,s,\xi)|^2\le C(1+\E|\xi|^2).$$
\end{itemize}

It is easy to see that the Assumption ({\bf A'}) implies Assumption ({\bf A}) so that Lemmas \ref{lem1} and \ref{lem2} are applicable under Assumption ({\bf A'}). Obviously, (a)
and (b) follows from Lemma \ref{lem1} and Lemma \ref{lem2}, respectively. 
Additionally, (c) holds true with $h$ being a negative exponential function by taking Lemma \ref{lem:stable1} into account. Finally, (d) is available with the aid of Lemma \ref{lem3} with $\vv_0=0$ therein. Therefore, we complete the associated proof. 
\end{proof}


\section{Proof of Theorem \ref{thm2} and Theorem \ref{thm3}}
\label{sec:ProofofInteractingParticSystem:PoC:thM2}

\begin{lemma}[Well posedness and moments]\label{lem4}
Let assumptions of Theorem \ref{th1} hold. 
Then, for each fixed $N\ge1$, the IPS 
\eqref{P2} has a unique strong solution $({\bf X}_{s,t}^{N,N})_{t\ge s}$. Moreover, there exist constants $C,\ll>0$ independent of $N$ such that for any $(s,t)\in\Delta$,  $\xi^N:=\text{vec}(\xi_1,\cdots,\xi_N),\eta^N:=\text{vec}(\eta_1,\cdots,\eta_N)\in L^2(\OO\to(\R^{d})^N,\F_s,\P)$  and $i\in\mathbb S_N,$
\begin{equation}\label{P6}
\E\big|{ X}_{s,t}^{i,N,\xi^N}\big|^2\le C\big(1+\e^{-\lambda(t-s)} \E|\xi_1|^2\big),
\end{equation}
and
\begin{equation}\label{P77}
\E\big|{ X}_{s,t}^{i,N,\xi^N}-{X}_{s,t}^{i,N,\eta^N}\big|^2\le C\e^{-\lambda(t-s)} \E|\xi_1-\eta_1|^2.
\end{equation}
\end{lemma}

\begin{proof}
For ${\bf x}^N:=\text{vec}\big(x_1,\cdots,x_N\big)\in(\R^{d})^N$ and $t\in\R$, 
let 
\begin{equation*}
{\bf b}_t({\bf x}^N) =\text{vec}\big(b_t(x_1,\hat\mu^{{\bf x}^N}),\cdots,b_t(x_N,\hat\mu^{{\bf x}^N})\big) \quad \mbox{ and } \quad \bm{\sigma}_t({\bf x}^N\big) =\mbox{diag}\big(\sigma_t(x_1,\hat\mu^{{\bf x}^N}),\cdots,\sigma_t(x_N,\hat\mu^{{\bf x}^N})\big),
\end{equation*}
where
$ \hat\mu^{{\bf x}^N}:=\frac{1}{N}\sum_{j=1}^N\delta_{x_j}\in\mathscr P(\R^d)$. 
With the notation ${\bf b}$, $\bm{\si}$ and ${\bf X}^{N,N}$ at hand, the IPS \eqref{P2} can be reformulated as an SDE on $(\R^{d})^N$,
\begin{equation}\label{P3}
\d {\bf X}_{s,t}^{N,N}={\bf b}_t({\bf X}_{s,t}^{N,N})\d t+\bm{ \si}_t({\bf X}_{s,t}^{N,N})\,\d  {\bf W}_t^N,\quad (s,t)\in \Delta,
\end{equation}
where 
$({\bf W}_t^N)_{t\ge s}:=\text{vec}\big(W_t^1,\cdots,W_t^N\big)_{t\ge s} $
is a  $dN$-dimensional Brownian motion. 
For ${\bf x}^N, {\bf y}^N\in(\R^{d})^N$, by the fact  
$$\frac{1}{N}\sum_{j=1}^N\delta_{(x_j,y_j)}\in \mathcal C(\hat\mu^{{\bf x}^N},\hat\mu^{{\bf x}^N}),$$
we then have
\begin{equation}\label{P8}
\mathbb W_2\left(\hat{\mu}^{\mathbf{x}^N}, \hat{\mu}^{\mathbf{y}^N}\right)^{2} \leq\frac{1}{N}\sum_{j=1}^N \int_{\R^d\times\R^d}|z-\tilde{z}|^2\delta_{(x_j,y_j)}(\d z,\d \tilde{z})= \frac{1}{N} |\mathbf{x}^N-\mathbf{y}^N |^{2}.
\end{equation}
This, together with \eqref{E20} and \eqref{E21}, yields that 
for any $t\in\R$ and ${\bf x}^N, {\bf y}^N\in(\R^{d})^N$, 
\begin{equation}\label{P4}
\begin{split}
    2&\big\<{\bf x}^N- {\bf y}^N,{\bf b}_t({\bf x}^N)-{\bf b}_t({\bf y}^N)\big\>+\|\bm{ \sigma}_t({\bf x}^N)-\bm{\sigma}_t({\bf y}^N)\|_{\rm HS}^2\\
    &=\sum_{j=1}^N\big(2\big\<  x_j-  y_j,  b_t\big(x_j,\hat{\mu}^{{\bf x}^N}\big)- b_t\big(y_j,\hat{\mu}^{{\bf y}^N}\big)\big\>+\big\|  \sigma_t\big(x_j,\hat{\mu}^{{\bf x}^N}\big)-  \sigma_t\big(y_j,\hat{\mu}^{{\bf y}^N}\big)\big\|_{\rm HS}^2\big)\\
    &\le K_1(t)|{\bf x}^N-{\bf y}^N|^2+N K_2(t)\mathbb W_2\big(\hat{\mu}^{{\bf x}^N},\hat{\mu}^{{\bf y}^N}\big)^2\\
    &\le \big(K_1(t)+K_2(t)\big)|{\bf x}^N-{\bf y}^N|^2,
\end{split}
\end{equation}
and that
\begin{equation}\label{P5}
\begin{split}
\| \sigma_t({\bf x}^N)- \sigma_t({\bf y}^N)\|_{\rm HS}^2
&
\le  K_3(t)\big(|{\bf x}^N-{\bf y}^N|^2+N\mathbb W_2\big(\hat\mu^{{\bf x}^N},\hat\mu^{{\bf y}^N}\big)^2\big)
\le 2K_3(t)|{\bf x}^N-{\bf y}^N|^2.
\end{split}
\end{equation}
Thus, with the help of \eqref{P4} and \eqref{P5} saying that we have the one-sided Lipschitz and the Lipschitz property as a system in $(\bR^{d})^N$, \eqref{P3} has  a unique strong solution (see e.g. \cite[Theorem 3.1.1, p.44]{prevot2007concise}) so the IPS \eqref{P2} also has a unique strong solution. At this level, quoting only \cite{prevot2007concise} and following exactly the proofs of Lemmas \ref{lem3} and \ref{lem:stable1} yields the estimates like \eqref{P6} and \eqref{P77}, where the RHS constants $C$ depend on the system's dimension $dN$ instead of just $d$. 
To avoid a repetition of proofs, we argue that the estimations showing that the constants $C$ are independent of $N$ are well-known (see \cite{salkeld2019LDPMV-SDEs,reis2018simulation}, and critically, we present a version of them in the proof of Lemma \ref{lem5}. For this reason, we shorten this proof -- the reader can verify that the proof is not circular.
\end{proof}

\begin{lemma}[Uniform-in-time  PoC]
\label{lem5}
Assume that assumptions of Theorem \ref{th1} holds. Then, there exist constants $\vv_0\in(0,1)$ and $C>0$ (independent of $N$)
such that for all $(t,s)\in\Delta$ and 
$\xi_i\in L^{2+\vv_0}(\OO\to\R^d,\F_s,\P)$ with $ i\in\mathbb S_N,$  
\begin{align}\label{P12}
 \mathbb{E} |X_{s,t}^{i,\xi_i}-X_{s,t}^{i,N,\xi_i}|^2 \leq C\varphi(N),
\end{align}
where $X_{s,t}^{i,\xi_i}$ and $X_{s,t}^{i,N,\xi_i}$ stand  for the solutions to \eqref{P1} and \eqref{P2} with $X_{s,s}^i=\xi_i
$
and $X_{s,s}^{i,N}=\xi_i,$ respectively, and $\varphi(\cdot)$ was introduced in \eqref{P7}.
\end{lemma}

\begin{proof}
Owing to \eqref{E23}, there exists a constant $\vv_*>0$ such that
\begin{equation}\label{RR}
-\ll_*:=\int_0^\tau   K_*(u)\,\d u\in(-\8,0)\qquad \mbox{ with } \qquad    K_*(u):=K_1(u)+(1+\vv_*)K_2(u).
\end{equation}
In the sequel, we write $X_{s,t}^i$
and $X_{s,t}^{i,N}$ instead of $X_{s,t}^{i,\xi_i}$ and $X_{s,t}^{i,N,\xi_i}$, respectively, for the sake of notation convenience. 
Let $\bar{\mu}^{ {\bf  X}^N}_{s,t}$ be the empirical measure corresponding to  ${\bf X}_{s,t}^N:=\big(X_{s,t}^1,\ldots,X_{s,t}^N\big)$, i.e., 
\begin{equation*}
 \bar{\mu}_{s,t}^{{\bf X}^N}(\d z)=\bar{\mu}^{{\bf X}^N_{s,t}}(\d z)=\frac{1}{N}\sum_{j=1}^N\delta_{X_{s,t}^i}(\d z),
\end{equation*}
and set 
$\Gamma_{s,t}^{i,N}:=X_{s,t}^{i}-X_{s,t}^{i,N}$.  
Next, applying  It\^o's formula followed by making use of \eqref{E20}  yields 
\begin{align}\label{eqn:ito_propogation}
\begin{split}
  \d\Big(\e^{-\int_s^t  K_*(u)\,\d u} |\Gamma_{s,t}^{i,N}|^2\Big) &
    =\e^{-\int_s^t  K_*(u)\,\d u}\big\{-  K_*(t)|\Gamma_{s,t}^{i,N}|^2+2\big\langle \Gamma_{s,t}^{i,N},b_t\big(X_{s,t}^{i}, \mu_{s,t} \big)-b_t\big({X}_{s,t}^{i,N}, \hat{\mu}^{{\bf X}^{N,N}}_{s,t} \big)\big\rangle\\ &\quad+\|\sigma_t\big(X_{s,t}^{i}, \mu_{s,t} \big)-\sigma_t\big({X}_{s,t}^{i,N}, \hat{\mu}^{{\bf X}^{N,N}}_{s,t} \big)\|_{\rm HS}^2\big\}\,\d t+\d M_t\\
    &\le\e^{-\int_s^t K_*(u)\,\d u}\big\{-( 1+\vv_*)K_2(t)|\Gamma_{s,t}^{i,N}|^2 +K_2(t)\mathbb W_2\big(\mu_{s,t},\hat{\mu}^{{\bf X}^{N,N}}_{s,t}\big)^2\big\}\,\d t+\d M_t\\
    &\le  \e^{-\int_s^t K_*(u)\,\d u}\big\{-( 1+\vv_*)K_2(t)|\Gamma_{s,t}^{i,N}|^2\\ &\qquad\qquad\quad\quad+K_2(t)\big(\mathbb W_2\big(\mu_{s,t},\bar{\mu}_{s,t}^{{\bf X}^N}\big)+\mathbb W_2\big(\bar{\mu}_{s,t}^{{\bf X}^N},\hat{\mu}^{{\bf X}^{N,N}}_{s,t}\big)\big)^2\big\}\,\d t+\d M_t\\
    &\le   \e^{-\int_s^t K_*(u)\,\d u}\big\{-( 1+\vv_*)K_2(t)|\Gamma_{s,t}^{i,N}|^2 +(1+\vv_*)K_2(t)\mathbb W_2\big(\bar{\mu}_{s,t}^{{\bf X}^N},\hat{\mu}^{{\bf X}^{N,N}}_{s,t}\big)^2\\
    &\quad\qquad\qquad\qquad+(1+1/\vv_*)K_2(t)\mathbb W_2\big(\mu_{s,t},\bar{\mu}_{s,t}^{{\bf X}^N}\big)^2\big\}\,\d t+\d M_t ,
\end{split}
\end{align}
for some martingale $(M_t)_{t\ge s}$,
where in the second inequality we applied the triangle inequality and in the last inequality we exploited the basic inequality: $(a+b)^2\le (1+\vv)a^2+(1+1/\vv)b^2$ for all $a,b\in\R$ and $\vv>0.$ From Lemma \ref{lem3}, there exist constants $\ll,C_1>0$ and $\vv_0\in(0,1)$ such that for all $(s,t)\in\Delta$ and $\xi\in L^{2+\vv_0}(\OO\to\R^d,\F_s,\P)$, 
\begin{equation}\label{P11}
  \int_{\R^d}|x|^{2+\vv_0}\,\mu_{s,t}(\d x)\le  C_1\big(1+\e^{-\lambda(t-s)}\E|\xi|^{2+
\vv_0}\big) .
\end{equation}
Next, 
according to \cite[Theorem 1]{fournier2015rate}, we find for some constant  $C_2>0$, 
\begin{equation*}
\E\mathbb W_2\big(\mu_{s,t},\bar{\mu}_{s,t}^{{\bf X}^N}\big)^2\le C_2\varphi(N)\bigg(\int_{\R^d}|x|^{2+\vv_0}\,\mu_{s,t}(\d x)\bigg)^{1+\vv_0/2},
\end{equation*}
where $\varphi(N)$ was defined in \eqref{P7}. Whence, by taking \eqref{P11} into consideration, it follows that for some constant $C_3>0,$
\begin{equation}\label{P10}
\E\mathbb W_2\big(\mu_{s,t},\bar{\mu}_{s,t}^{{\bf X}^N}\big)^2\le C_3\varphi(N).
\end{equation}
Furthermore, in the light  of \eqref{P8}, we have 
\begin{equation}\label{P9}
\E\mathbb W_2(\bar{\mu}_{s,t}^{{\bfX}^N},\hat{\mu}^{{\bf X}^{N,N}}_{s,t})^2\le \frac{1}{N}\sum_{j=1}^N\E|\Gamma_{s,t}^{j,N}|^2=\E|\Gamma_{s,t}^{i,N}|^2, \quad i\in\mathbb S_N,
\end{equation}
where the identity holds true since, for all $i\neq j$, $\Gamma_{s,t}^{j,N}$ and $\Gamma_{s,t}^{i,N}$
share the same distribution. Now, combining \eqref{P10} with \eqref{P9} and taking $X_{s,s}^i=X_{s,s}^{i,N}=\xi_i$ into account, we deduce from \eqref{eqn:ito_propogation} that for some constant $C_4>0,$
\begin{equation*}
\begin{split}
\E|\Gamma_{s,t}^{i,N}|^2&\le (1+1/\vv_*)\int_s^t \e^{ \int_u^t  K_*(v)\,\d v} K_2(u)\E\mathbb W_2\big(\mu_{s,u},\bar{\mu}_{s,u}^{{\bf X}^N}\big)^2\,\d u\\
&\le C_4\varphi(N)\int_s^t \e^{ \int_u^t  K_*(v)\,\d v} K_2(u)\,\d u.
\end{split}
\end{equation*}
Consequently, the assertion \eqref{P12} follows from \eqref{RR} and the periodic property of $K_1,K_2.$
\end{proof}

Below, we intend to complete the proof of Theorem \ref{thm2}.

\begin{proof}[Proof of Theorem \ref{thm2}]
Owing to Lemma \ref{lem4}, the mapping $\phi^N$ defined in \eqref{1T} admits the stochastic semi-flow property:
$$
\phi^N(s,t,\xi^N,\omega)=\phi(r,t,\phi(s,r,\xi^N,\omega),\omega),
$$ for all $s\le r\le t$, $\xi^N\in L^2(\OO\to(\R^{d})^N,\F_s,\P)$ and $\omega\in\Omega$. Note that \eqref{P3}
is a classical time-inhomogeneous SDE with $\tau$-periodic coefficients ${\bf b}$ and $\bm{\si}$. Then, we conclude that 
$$
\phi^N(t+\tau,s+\tau,\xi^N,\omega)=\phi^N(t,s,\xi^N,\theta_\tau\omega),
$$
for all $(s,t)\in\Delta$ and $\xi^N\in L^2(\OO\to(\R^{d})^N,\F_s,\P)$, $\bP$-a.s. Next, in addition to 
\eqref{P6} and \eqref{P77}, via the general criteria on existence of random periodic solutions for stochastic dynamical systems (see, for instance, \cite[Proposition 3.6]{bao2022random}), the stochastic semi-flow $\phi^N$,  defined in  \eqref{1T}, has a unique pathwise random $\tau$-periodic solution $ {\bf X}^{*,N,N}_t \in L^2(\Omega\to(\R^{d})^N,\F_t,\P)$ satisfying \eqref{P14} and \eqref{P15}.
\end{proof}

\begin{proof}[Proof of Theorem \ref{thm3}]
Lemma \ref{lem5} states the PoC result of the IPS \eqref{P2} to the non-interacting particle system \eqref{P1} (which coincides with the McKean-Vlasov SDE).

It remains only to show the PoC for the IPS's pathwise random periodic solution to the corresponding one of the non-interacting particle system. 
Below, let $\xi_i\in\R^d,i\in\mathbb S_N,$
be deterministic. 
It is easy to see that for all $t\ge s  $  and $i\in\mathbb S_N$,
\begin{equation}\label{P18}
\E|X_t^{*,i}-{X}_t^{*,i,N}|^2\le 3\E|X_t^{*,i}-{ X}_{s,t}^{i,\xi_i}|^2+3\E|X_{s,t}^{i,\xi_i}-{ X}_{s,t}^{i,N,\xi_i}|^2+3\E| { X}_{s,t}^{i,N,\xi_i}-{ X}_t^{*,i,N}|^2.
\end{equation}
Due to \eqref{P17} with $X_{s,t}^\xi$ and $X_t^*$ therein replaced by $X_{s,t}^{i,\xi_i}$ and $X_t^{i,*}$, respectively, we conclude that the first term on the RHS of \eqref{P18} goes to zero as $s\to-\8$. 
Furthermore, by taking advantage of \eqref{P15}, we infer that the third term on the RHS of \eqref{P18} approaches zero as $s\to-\8$. Finally, the assertion \eqref{P19} follows by handling the second term on the RHS of \eqref{P18} via Lemma \ref{lem5} with $\xi_i\in\R^d$ being deterministic.
\end{proof}

 \section{Proofs of Theorem \ref{thm*} to Theorem \ref{thm6}}\label{sec:partial}

\begin{lemma}[Well-posedness and uniform moment bounds]\label{lem5.1}
Assume  Assumption ({\bf H}) with $K_1>2K_2$ and $\int_0^\tau\alpha_u\,\d u\in(0,\8)$. Then,  the McKean-Vlasov SDE \eqref{*X} has a unique strong solution $(X_{s,t})_{t\ge s}$ and there exists a universal constant $C>0$ such that for all  $X_{s,s}\in L^2(\OO\to\R^d,\F_s,\P)$,
\begin{equation}\label{**}
\sup_{t\ge s}\E|X_{s,t}|^2\le  C\big(1+\E|X_{s,s}|^2\big).
\end{equation}
\end{lemma}

\begin{proof}
By  \eqref{XX}, we deduce  that 
for all $t\in\R$ and $x,y\in\R^d$,  $\mu,\nu\in\mathscr P_1(\R^d)$,  
 \begin{equation*} 
 \begin{split}
\<(\tilde{b}_t*\mu)(x)-(\tilde{b}_t*\nu)(y),x-y\>&=\int_{\R^d\times\R^d}\<\tilde b_t(x,z_1)-\tilde b_t(x,z_1),x-y\>\pi(\d z_1,\d z_2)\\
&\le K_2\alpha_t|x-y|\int_{\R^d\times\R^d}\big(|x-y|+|z_1-z_2|\big)\pi(\d z_1,\d z_2),
\end{split}
\end{equation*}
where $\pi\in\mathscr C(\mu,\nu)$. Thus, taking infimum w.r.t. $\pi$ on both sides yields that  for all $t\in\R$ and $x,y\in\R^d$,  $\mu,\nu\in\mathscr P_1(\R^d)$,  
\begin{equation}\label{PP}
 \begin{split}
\<(\tilde{b}_t*\mu)(x)-(\tilde{b}_t*\nu)(y),x-y\>\le K_2\alpha_t\big(|x-y|+\mathbb W_1(\mu,\nu)\big)|x-y|.
\end{split}
\end{equation}
Let $b_t(x,\mu)=\hat b_t(x)+(\tilde b_t*\mu)(x)$ for $t\in\R,x\in\R^d$ and $\mu\in\mathscr P(\R^d)$. Then, 
combining \eqref{X3}  with \eqref{PP}  enables us to derive  that for all $t\in\R$,  $x,y\in\R^d$ and $\mu,\nu\in\mathscr P_1(\R^d)$, 
 \begin{equation}\label{XXX}
 \begin{split}
&\<b_t(x,\mu)-b_t(y,\nu),x-y\>+\frac{1}{2}\|\hat \sigma_t(x)-\hat \sigma_t(y)\|_{\rm HS}^2\\
&\le \alpha_t\big((K_0+K_1)|x-y|^2\I_{\{|x-y|\le\ell_0\}}-(K_1-K_2) |x-y|^2\big) +K_2\alpha_t \mathbb W_1(\mu,\nu) |x-y|,
\end{split}
\end{equation} 
Whence,  according to   \cite[Theorem 2.1]{wang2018distribution}, the McKean-Vlasov SDE \eqref{*X} under investigation has a unique strong solution.

 By taking advantage of \eqref{XX} and \eqref{XXX}, in addition to the fact that  $\mathbb W_1\le \mathbb W_2$,  for any $\vv>0,t\in\R, x\in\R^d$ and $\mu\in\mathscr P_2(\R^d)$,
there exists a positive $\tau$-periodic function $\bb_\vv(\cdot)$ independent of $x$ such that
\begin{equation}
\begin{split}
    2\<b_t(x,\mu),x\>+\| \hat \sigma_t(x) \|_{\rm HS}^2&=2\<b_t(x,\mu)-b_t({\bf0},\delta_{\bf0}),x\>+\|\hat \sigma_t(x)-\hat \sigma_t({\bf0})\|_{\rm HS}^2\\
    &\quad+2\< b_t({\bf0},\delta_{\bf0}),x\>+2\<\hat \sigma_t(x)-\hat \sigma_t({\bf0}),\hat \sigma_t({\bf0})\>_{\rm HS}+\| \sigma_t({\bf0})\|_{\rm HS}^2\\
    &\le 2\alpha_t\big((K_0+K_1)|x|^2\I_{\{|x|\le\ell_0\}}-(K_1-K_2) |x|^2\big) +2K_2\alpha_t \mathbb W_2(\mu,\delta_{\bf0}) |x|\\
    &\quad+2\< b_t({\bf0},\delta_{\bf0}),x\>+2\<\hat \sigma_t(x)-\hat \sigma_t({\bf0}),\sigma_t({\bf0})\>_{\rm HS}+\| \hat \sigma_t({\bf0})\|_{\rm HS}^2\\
    &\le \beta_\vv(t)-2\alpha_t(K_1-3K_2/2-\vv)|x|^2+K_2\alpha_t\mathbb W_2(\mu,\delta_{\bf0})^2.
    \end{split}
\end{equation}
Consequently,
from It\^o's formula and the basic fact that  $\mathbb W_2(\mu,\delta_{\bf0})^2= \mu(|\cdot|^2)$ for $\mu\in\mathscr P_2(\R^d)$,  it follows that for all $t\ge s$ and $\vv>0,$
\begin{equation*}
\e^{2(K_1-2K_2-\vv)\int_s^t\alpha_u\,\d u}\E|X_{s,t}|^2\le \E|X_{s,s}|^2+\int_s^t\e^{2(K_1-2K_2-\vv)\int_s^u\alpha_v\,\d v}\bb_\vv(u)\,\d u. 
\end{equation*}
In particular,  choosing  $\vv=(K_1-2K_2)/2>0$ due to $K_1>2K_2$ and taking the $\tau$-periodic properties of $\alpha_\cdot$ and $\bb_\vv(\cdot)$, and $\int_0^\tau\alpha_u\,\d u\in(0,\8)$ into consideration results in the desired assertion \eqref{**}. 
\end{proof}

 The following Lemma demonstrates that the distribution flow starting from different initial distributions is weakly $\mathbb W_1$-contractive once the perturbation intensity is small (i.e., the corresponding Lipschitz constant of $\tilde b_t$ is small enough).

 \begin{lemma}[Exponential ergodicity]\label{lem5.2}
Assume   ({\bf H}) with $\int_0^\tau\alpha_u\,\d u>0$. Then,  there exist constants $C,\lambda,K_2^*>0$ such that for all $(s,t)\in\Delta$,
$K_2\in(0,K_2^*]$ and  $\mu,\nu\in\mathscr P_1(\R^d)$,
\begin{equation}\label{J*}
 \mathbb W_1\big(\mathscr L_{X_{s,t}^\mu},\mathscr L_{X_{s,t}^\nu}\big)\le C\e^{-\lambda(t-s)} \mathbb W_1( \mu, \nu),
 \end{equation}
 where $\mathscr L_{X_{s,t}^\mu}$ stands for the law of $X_{s,t}$ , the solution to \eqref{*X},  with $\mathscr L_{X_{s,s}}=\mu$.
 
 \end{lemma}
 
 \begin{proof} 
In the sequel,    we consider the following  decoupled SDE associated with \eqref{*X}: for any  $ (s,t)\in\Delta$,
 \begin{equation} 
\d Y_{s,t}^{\nu,\mu}=\big(\,\hat b_t(Y_{s,t}^{\nu,\mu})+(\tilde b_t*\mu_{s,t})(Y_{s,t}^{\nu,\mu})\big)\,\d t+\sqrt{\alpha_t}\d B_t+\hat \si_t(Y_{s,t}^{\nu,\mu})\,\d W_t, \quad Y_{s,t}^{\nu,\mu}\sim\nu, 
\end{equation}
where $\mu_{s,t}:=\mathscr L_{X_{s,t}^\mu}$.  Due to the strong well-posedness (so weak well-posedness)  of \eqref{*X}, we have $\mu_{s,t}=\mathscr L_{Y_{s,t}^{\mu,\mu}}$. 

 Below, to avoid the singularity on the diagonal, we introduce the following cut-off function: 
 for  each $\vv>0$,  
\begin{equation}\label{*mm*}
\phi_\vv(r):=
\begin{cases}
0 &, r\in[0,\frac{5\vv}{8}], 
\\
1-384\vv^{-3}\big(\frac{1}{3}\big(r-\frac{7\vv}{8}\big)+\frac{\vv}{8}\big)\big(r-\frac{7\vv}{8}\big)^2 &, \frac{5\vv}{8}\le r\le \frac{7\vv}{8},
\\
1 &, r\ge\frac{7\vv}{8}.
\end{cases}
\end{equation}
Obviously,   $\phi_\vv\in C^1([0,\8);[0,1])$. 
With the cut-off function $\phi_\vv$ at hand,  we shall work with the following coupled SDE: 
 \begin{equation} \label{EWW*}
 \begin{cases}
\d \bar Y_{s,t}^{\mu}= b_t(\bar Y_{s,t}^{\mu},\mu_{s,t})\,\d t+\sqrt{\alpha_t}\phi_\vv(|Z_{s,t}|)^{\frac{1}{2}}\d \bar B_t\\
\qquad\qquad+\sqrt{\alpha_t}\big(1-\phi_\vv(|Z_{s,t}|)\big)^{\frac{1}{2}}\,\d\hat B_t+\hat \si_t(\bar Y_{s,t}^{\mu})\,\d W_t, \quad \bar Y_{s,s}^\mu\sim\mu\\
\d \hat Y_{s,t}^{\nu}=b_t(\hat Y_{s,t}^{\nu},\nu_{s,t})\,\d t+\sqrt{\alpha_t}\phi_\vv(|Z_{s,t}|)^{\frac{1}{2}}\Pi(Z_{s,t})\d \bar B_t\\
\qquad\qquad+\sqrt{\alpha_t}\big(1-\phi_\vv(|Z_{s,t}|)\big)^{\frac{1}{2}}\,\d \hat B_t+\hat \si_t(\hat Y_{s,t}^{\nu})\,\d W_t, \quad \hat Y_{s,s}^\nu\sim\nu,\\
\end{cases}
\end{equation}
 where $b_t(x,\mu):=\hat b_t(x)+(\tilde b_t*\mu)(x)$ for $t\in\R,x\in\R^d$ and $\mu\in\mathscr P(\R^d)$; $Z_{s,t}:=\bar Y_{s,t}^{\mu}-\hat Y_{s,t}^{\nu}$; 
 $(\bar B_t)_{t\ge s}$ and  $(\hat B_t)_{t\ge s}$,  independent of $(W_t)_{t\ge0}$,
 are mutually dependent $d$-dimensional Brownian motions;    for ${\bf n}(x)=\frac{x}{|x|}\I_{\{x\neq{\bf0}\}}$,
 \begin{align}\label{EEY}
 \Pi(x):=I_{d\times d}-2{\bf n}(x){\bf n}(x)^*\I_{\{x\neq {\bf0}\}}.
 \end{align} 
In terms of L\'{e}vy's characterization for Brownian motions, we conclude that $(\bar Y_{s,t}^\mu,\hat Y_{s,t}^\nu)_{t\ge s}$ is a coupling of $(Y_{s,t}^{\mu,\mu}, Y_{s,t}^{\nu,\nu})_{t\ge s}$. Owing to   existence of optimal couplings, in the analysis below, we shall choose the initial values $\bar Y_{s,s}^\mu$
and $\hat Y_{s,s}^\nu$ such that
\begin{align}\label{KL}
\E|\bar Y_{s,s}^\mu-\hat Y_{s,s}^\nu|= \mathbb W_1(\mu,\nu) .
\end{align}

For any $\delta>0$ and $x\in\R^d$, set $V_\delta(x):=(\delta+|x|^2)^{\frac{1}{2}}$. A direct calculation shows that for any $\vv>0,$
\begin{align}\label{EY}
\frac{x}{V_\delta(x)}\overset{\delta\to0}{\longrightarrow}{{\bf n}(x)}\quad\mbox{ and } \quad  \frac{\delta\phi_\vv(|x|)}{V_\delta(x)^3}\le \frac{\delta}{(\delta+\vv^2/4)^{\frac{3}{2}}}\overset{\delta\to0}{\longrightarrow}0.
\end{align}
Note trivially from \eqref{EWW*} that
\begin{align*}
    Z_{s,t}=&\big(  b_t(\bar Y_{s,t}^{\mu},\mu_{s,t})- b_t(\hat Y_{s,t}^{\nu},\nu_{s,t}) \big)\,\d t+2\ss{\alpha_t}\phi_\vv(|Z_{s,t}|)^{\frac{1}{2}}{\bf n}(Z_{s,t}){\bf n}(Z_{s,t})^*\,\d\bar B_t\\
    &\quad+\big(\hat \si_t(\bar Y_{s,t}^{\mu})-\hat \si_t(\hat Y_{s,t}^{\nu})\big)\,\d W_t.
\end{align*}
 Applying It\^o's formula and taking \eqref{XXX} into consideration enables us to derive that 
 \begin{align*}
    \d V_\delta(Z_{s,t}) \le&\bigg(\frac{1}{V_\delta(Z_{s,t})}\<Z_{s,t},b_t(\bar Y_{s,t}^{\mu},\mu_{s,t})- b_t(\hat Y_{s,t}^{\nu},\nu_{s,t})\>  +\frac{2\delta\alpha_t\phi_\vv(|Z_{s,t}|)}{V_\delta(Z_{s,t})^3} +\frac{1}{2V_\delta(Z_{s,t})}\|\hat \si_t(\bar Y_{s,t}^{\mu})-\hat \si_t(\hat Y_{s,t}^{\nu})\|_{\rm HS}^2\bigg)\,\d t\\
&\quad+\frac{2\ss{\alpha_t}\phi_\vv(|Z_{s,t}|)^{\frac{1}{2}}}{V_\delta(Z_{s,t}) }\<Z_{s,t}, \d\bar B_t\> +\frac{1}{V_\delta(Z_{s,t})}
    \<Z_{s,t}, (\hat \si_t(\bar Y_{s,t}^{\mu})-\hat \si_t(\hat Y_{s,t}^{\nu}) )\,\d W_t\>\\
    &\le \frac{\alpha_t}{V_\delta(Z_{s,t}) }\big((K_0+K_1)|Z_{s,t}|^2\I_{\{|Z_{s,t}|\le\ell_0\}}-(K_1-K_2)|Z_{s,t}|^2+K_2\mathbb W_1(\mu_{s,t},\nu_{s,t})|Z_{s,t}|\big)\,\d t\\
  &\quad+  \frac{2\delta\alpha_t\phi_\vv(|Z_{s,t}|)}{V_\delta(Z_{s,t})^3}\,\d t+\frac{2\ss{\alpha_t}\phi_\vv(|Z_{s,t}|)^{\frac{1}{2}}}{V_\delta(Z_{s,t}) }\<Z_{s,t}, \d\bar B_t\> \\
&\quad+\frac{1}{V_\delta(Z_{s,t})}
    \<Z_{s,t}, (\hat \si_t(\bar Y_{s,t}^{\mu})-\hat \si_t(\hat Y_{s,t}^{\nu}) )\,\d W_t\>.
 \end{align*}
 This, together with \eqref{EY}, leads to 
  \begin{equation}\label{EEE*}
  \begin{split}
    \d |Z_{s,t}| \le&\alpha_t\big(   \varphi(|Z_{s,t}|)+K_2\mathbb W_1(\mu_{s,t},\nu_{s,t}) \big)\,\d t\\
    &\quad+ 2\ss{\alpha_t}\phi_\vv(|Z_{s,t}|)^{\frac{1}{2}} \<{\bf n}(Z_{s,t}), \d\bar B_t\> + 
    \<{\bf n}(Z_{s,t}), (\hat \si_t(\bar Y_{s,t}^{\mu})-\hat \si_t(\hat Y_{s,t}^{\nu}) )\,\d W_t\>,
 \end{split}
 \end{equation}
where for any $r\ge0,$
\begin{align*}
    \varphi(r):=(K_0+K_1)r\I_{\{r\le\ell_0\}}-(K_1-K_2)r.
\end{align*}

 Define the following  function:  
 \begin{align}\label{YY}
     f(r)=c_1r+\int_0^r\e^{-c_2u}\,\d u,\quad r\ge0,
 \end{align}
 where
$  c_1:=\e^{-c_2\ell_0}$ and  $c_2:=2(K_0+K_2)\ell_0$.
 It is easy to see that  for any $r\ge0,$
\begin{align}\label{EY1}
    f'(r)=c_1+\e^{-c_2r}>0\quad \mbox{ and }\quad  f''(r)=-c_2\e^{-c_2r}<0.
\end{align}
 Once more,  for the parameter 
 $$c_*:=\frac{c_1((2(K_0+K_2))\wedge (K_1-K_2))}{1+c_1},
 $$
 applying It\^o's formula gives that 
 \begin{align*}
     \d\Big (\e^{c_*\int_s^t\alpha_r\,\d r}f(|Z_{s,t}|)\Big)\le \alpha_t\e^{c_*\int_s^t\alpha_r\,\d r}\big(&c_*f(|Z_{s,t}|)+\psi_\vv(|Z_{s,t}|)+K_2f'(|Z_{s,t}|) \mathbb W_1(\mu_{s,t},\nu_{s,t}) \big)\,\d t+\d M_{s,t},
 \end{align*}
 for some martingale $(M_{s,t})_{t\ge s}$, in which  for any $r\ge0$,
 \begin{align*}
     \psi_\vv(r):=f'(r)\varphi(r)+2f''(r)\phi_\vv(r).
 \end{align*}

It is ready to see  from \eqref{EY1}   that  for $K_1>K_2$ and some constant $c_0>0,$
\begin{align*}
    \psi_\vv(r)
    &\le\big((c_1+\e^{-c_2 r})(K_0+K_2)r-2c_2\e^{-c_2r}\big)\phi_\vv(r)\I_{\{r\le\ell_0\}}-c_1(K_1-K_2)r\I_{\{r>\ell_0\}}\\
    &\quad+(c_1+\e^{-c_2 r})(K_0+K_2)r(1-\phi_\vv(r))\I_{\{r\le\ell_0\}}\\
    &\le 2\e^{-c_2r}\big( (K_0+K_2)r- c_2 \big)\phi_\vv(r)\I_{\{r\le\ell_0\}}-c_1(K_1-K_2)r\I_{\{r>\ell_0\}}\\
    &\quad+2(K_0+K_2)r(1-\phi_\vv(r))\I_{\{r\le\ell_0\}}\\
    &\le -\frac{c_2c_1}{\ell_0}r\I_{\{r\le\ell_0\}}-c_1(K_1-K_2)r\I_{\{r>\ell_0\}}+c_0r(1-\phi_\vv(r))\I_{\{r\le\ell_0\}}\\
    &\le-c_*f(r)+c_0r(1-\phi_\vv(r))\I_{\{r\le\ell_0\}},
\end{align*}
where in the second inequality we used the fact that $c_1\le \e^{-c_2r}$ for $r\le\ell_0$,  and  the third inequality and the last inequality hold  true due to the alternative of $c_2$. Whereafter, by invoking the property of $\phi_\vv$, we have for any $r\ge0,$
\begin{align*}
    \limsup_{\vv\to0} \psi_\vv(r)\le -c_*f(r).
\end{align*}
This further implies that 
\begin{align}\label{EWW}
     \d\Big (\e^{c_*\int_s^t\alpha_r\,\d r}f(|Z_{s,t}|)\Big)\le K_2(1+c_1)\alpha_t\e^{c_*\int_s^t\alpha_r\,\d r} \mathbb W_1(\mu_{s,t},\nu_{s,t}) \,\d t+\d M_{s,t}.
 \end{align}
Recall that $(\bar Y_{s,t}^\mu,\hat Y_{s,t}^\nu)_{t\ge s}$ is a coupling of $(Y_{s,t}^{\mu,\mu}, Y_{s,t}^{\nu,\nu})_{t\ge s}$,  $\mu_{s,t}=\mathscr L_{Y_{s,t}^{\mu,\mu}}$, and $\nu_{s,t}=\mathscr L_{Y_{s,t}^{\nu,\nu}}$. Thus, in addition to $\|f'\|_\8\le 1+c_1$, 
 we deduce that 
\begin{align*}
\e^{c_*\int_s^t\alpha_r\,\d r}\mathbb W_1(\mu_{s,t},\nu_{s,t}) &\le\e^{c_*\int_s^t\alpha_r\,\d r}\E f(|Z_{s,t}|)\\
&\le \E f(|Z_{s,s}|)+K_2(1+c_1)\int_s^t\alpha_u\e^{c_*\int_s^u\alpha_r\,\d r} \mathbb W_1(\mu_{s,u},\nu_{s,u}) \,\d u.
\end{align*}
Subsequently,  the Gronwall inequality yields that 
\begin{align*}
\mathbb W_1(\mu_{s,t},\nu_{s,t}) 
&\le \E f(|Z_{s,s}|)\e^{-(c_*-K_2(1+c_1))\int_s^t\alpha_r\,\d r}.
\end{align*}
Obviously, there exists a constant $K_2^*>0$ such that  $K_1>K_2$ and $c_*>K_2(1+c_1)$
for any $K_2\in(0,K_2^*]$.  As a consequence, the desired assertion \eqref{J*} follows directly by taking \eqref{KL} into consideration. 
 \end{proof}

Below, we aim to address the issue on the uniform-in-time PoC   for the non-interacting particle systems and the IPSs associated with the McKean-Vlasov SDE \eqref{*X}. 
So far, there are considerable literature on the PoC in a  finite horizon; see the  monographs  \cite{carmona2018probabilistic} and \cite{sznitman1991topics}. In contrast, the research on the PoC in an  infinite horizon is scarce. 
Recently, there is a great progress on the uniform-in-time  PoC   for IPSs corresponding to granular media type SDEs with additive noises; see the excellent work \cite{durmus2020elementary}. In the following context, by following essentially the idea in \cite{durmus2020elementary}, we proceed to tackle the issue on the uniform-in-time PoC   for  stochastic IPSs concerned with time-inhomogenenous McKean-Vlasov SDE \eqref{*X} with 
multiplicative noise. Furthermore,  it is worthy to point out that, to handle the singularity at the zero point, the auxiliary function $\phi_\vv$ defined in \eqref{*mm*} below plays a crucial role in the  analysis to be implemented.

\begin{lemma}[Uniform-in-time PoC]\label{lem5.3}
Assume ({\bf H}) with $ \int_0^\tau\alpha_u\,\d u\in(0,\8)$, and suppose further $\E|X_{s,s}^{i,N}|^2<\8$ for each $i\in\mathbb S_N$.
Then,  there exist constants $C_1, C_2,\ll, K_2^*>0$ such that for all $K_2\in(0,K_2^*]$, $N\ge2$, and $(s,t)\in \Delta$,
\begin{align}\label{L4}
    \mathbb W_1 \Big( \mathscr L_{X_{s,t}^i}  , \mathscr L_{X_{s,t}^{i,N}}  \Big) \leq C_1\e^{-\ll (t-s)}\mathbb W_1 \Big( \mathscr L_{X_{s,s}^i}  , \mathscr L_{X_{s,s}^{i,N}} \Big)  +\frac{C_2}{\sqrt{N}},
\end{align}
where $C_1,C_2$ is independent of $(s,t)\in\Delta$.
\end{lemma}

\begin{proof} 
To begin, we recall that $\phi_\vv$ and $\Pi$ were defined respectively in \eqref{*mm*} and \eqref{EEY}. 
Consider the following SDE  for all $i\in\mathbb S_N$, $(t,s)\in\Delta)$, and $\vv>0,$
\begin{equation}\label{P1--}
\begin{cases}
\d \bar X_{s,t}^{i,\vv}=\big(\,\hat{b}_t(\bar X_{s,t}^{i,\vv})+(\tilde{b}_t*\bar\mu_{s,t}^{i,\vv})(\bar X_{s,t}^{i,\vv})\big)\,\d t+\ss{\alpha_t}\phi_\vv(|\bar Z_{s,t}^{i,N,\vv}|)^{\frac{1}{2}}\,\d B_t^{*,i}\\
~~~~~~~~~~~~~~+\ss{\alpha_t}\big(1-\phi_\vv(|\bar Z_{s,t}^{i,N,\vv}|)\big)^{\frac{1}{2}}\,\d \hat B_t^{*,i}+\hat{\sigma}_t(\bar X_{s,t}^{i,\vv})\,\d W_t^i,\quad \quad \bar X_{s,s}^{i,\vv}\sim X_{s,s}^i,\\
\d \bar X_{s,t}^{i,N,\vv}
=
\big(\,\hat{b}_t(\bar X_{s,t}^{i,N,\vv})+\frac{1}{N}\sum_{j=1}^N\tilde{b}_t(\bar X_{s,t}^{i,N,\vv},\bar X_{s,t}^{j,N,\vv})
\big)\,\d t 
+\ss{\alpha_t}\phi_\delta(|\bar Z_{s,t}^{i,N,\delta}|)^{\frac{1}{2}}\Pi(\bar Z_{s,t}^{i,N,\delta})\,\d B_t^{*,i}\\
~~~~~~~~~~~~~~+\ss{\alpha_t}\big(1-\phi_\vv(|\bar Z_{s,t}^{i,N,\vv}|)\big)^{\frac{1}{2}}\,\d \hat B_t^{*,i}+\hat{\si}_t(\bar X_{s,t}^{i,N,\vv})\,\d W_t^i,\quad \quad \bar X_{s,s}^{i,N,\vv}\sim X_{s,s}^{i,N},
\end{cases}
\end{equation}
where $\bar \mu_{s,t}^{i,\vv}:=\mathscr L_{\bar X_{s,t}^{i,\vv}}$, $\bar Z_{s,t}^{i,N,\vv}:=\bar X_{s,t}^{i,\vv}-\bar X_{s,t}^{i,N,\vv},$
and $(B^{*,i}_t)_{t\ge s}$ and  $(\hat{B}^{*,i}_t)_{t\ge s}$  are  independent copies of $(B^{i}_t)_{t\ge s}$. We further assume that $(\bar X_{s,s}^{i,\vv}, \bar X_{s,s}^{i,N,\vv}), i\in\mathbb S_N,$ are i.i.d random variables with finite second moments.

In the sequel, for the sake of notation brevity, we shall write $\bar\mu_{s,t}^i$,
$\bar X_{s,t}^i$, $\bar X_{s,t}^{i,N}$ and $\bar Z_{s,t}^{i,N}$ instead of $\bar\mu_{s,t}^{i,\vv}$,
$\bar X_{s,t}^{i,\vv}$, $\bar X_{s,t}^{i,N,\vv}$ and $\bar Z_{s,t}^{i,N,\vv}$, respectively.
Set for any $(t,s)\in\Delta,$
\begin{align*}
\bar{ {\bf X}}_{s,t}^N =(\bar X_{s,t}^1,\cdots,\bar X_{s,t}^N)\quad \mbox{ and } \quad\bar{ {\bf X}}_{s,t}^{N,N}=(\bar X_{s,t}^{1,N},\cdots, \bar X_{s,t}^{N,N}).
\end{align*}
Via   L\'{e}vy's characterisation for Brownian motions, we infer that $(\bar{ {\bf X}}_{s,t}^N,\bar{ {\bf X}}_{s,t}^{N,N})_{t\ge s}$ is a coupling of $({\bf X}_{s,t}^N, {\bf X}_{s,t}^{N,N})_{t\ge s}$, determined by \eqref{P1-} and \eqref{P2-}.

By following the strategy to derive \eqref{EEE*}, we deduce that 
 \begin{equation} 
  \begin{split}
    \d |\bar Z_{s,t}^{i,N}| \le&\alpha_t\big(   \varphi(|\bar Z_{s,t}^{i,N}|)+ |\Theta_{s,t}^{i,N} |\big)\,\d t\\
    &\quad+ 2\ss{\alpha_t}\phi_\vv(|\bar Z_{s,t}^{i,N}|)^{\frac{1}{2}} \<{\bf n}(\bar Z_{s,t}^{i,N}, \d\bar B_t^{*,i}\> + 
    \<{\bf n}(\bar Z_{s,t}^{i,N}), (\hat \si_t(\bar Y_{s,t}^{\mu})-\hat \si_t(\hat Y_{s,t}^{\nu}) )\,\d W_t^i\>,
 \end{split}
 \end{equation}
where  $\varphi(r):=(K_0+K_1)r\I_{\{r\le\ell_0\}}-K_1r, r\ge0,$ and 
$$ \Theta_{s,t}^{i,N}:=(\tilde{b}_t*\bar\mu_{s,t}^{i})(\bar X_{s,t}^{i})-\frac{1}{N}\sum_{j=1}^N\tilde{b}_t (\bar X_{s,t}^{i,N},\bar X_{s,t}^{j,N}).$$

Let the function $f$ be defined as in \eqref{YY} with $c_1:=\e^{-c_2\ell_0}$ and $c_2=2(K_0+K_1)\ell_0$. Then, 
by  mimicking the procedure to obtain \eqref{EWW}, for $c_*:=\frac{K_1c_1}{1+c_1}$,
we find that 
\begin{align}\label{ER*} 
  \e^{c_*\int_s^t\alpha_r\,\d r}\E f(|\bar Z_{s,t}^{i,N}|) \le \E f(|\bar Z_{s,s}^{i,N}|) 
  +(1+c_1) \int_s^t\e^{c_*\int_s^u\alpha_r\,\d r} |\Theta_{s,u}^{i,N} |\,\d u.
 \end{align}

Set for any $i\in\mathbb S_N$ and $(s,t)\in\Delta$, 
$$\Upsilon_{s,t}^i:=\bigg(\E\Big|(\tilde{b}_t*\bar \mu_{s,t}^{i})(\bar X_{s,t}^{i})-\frac{1}{N}\sum_{j=1}^N\tilde{b}_t(\bar X_{s,t}^{i},\bar X_{s,t}^{j})\Big|^2\bigg)^{\frac{1}{2}}.$$
Recall that $(\bar X_{s,s}^{i,\vv}, \bar X_{s,s}^{i,N,\vv}), i\in\mathbb S_N,$ are i.i.d random variables. Thus, 
by H\"older's inequality, it follows from \eqref{XX} and $f(r) \ge c_1r$ for all $r\ge0 $ that 
\begin{equation}\label{L2}
\begin{split}
\E|\Theta_{s,t}^{i,N}|
&\le \Upsilon_{s,t}^i+ \frac{1}{N}\sum_{j=1}^N\E\big|\tilde{b}_t(\bar X_{s,t}^{i},\bar X_{s,t}^{j})-\tilde{b}_t(\bar X_{s,t}^{i,N},\bar X_{s,t}^{j,N})\big|\\
&\le \Upsilon_{s,t}^i+\frac{K_2\alpha_t}{N}\sum_{j=1}^N\big(\E|\bar Z_{s,t}^{i,N}|+\E|\bar Z_{s,t}^{j,N}|\big)\\
&\le  \Upsilon_{s,t}^i+\frac{2 K_2\alpha_t}{c_1}\E f(|\bar Z_{s,t}^{i,N}|),
\end{split}
\end{equation}
 where   the third inequality  holds true since $\bar Z_{s,t}^{i,N}$ and $\bar Z_{s,t}^{j,N}$ enjoy the same law.

Note that, for $i\neq j,$  $\bar X_{s,t}^i$ and $\bar X_{s,t}^j$ are i.i.d. stochastic processes so $\bar\mu_{s,t}^i=\bar\mu_{s,t}^j$. Then,  by following exactly the line to achieve \cite[(28)]{durmus2020elementary}, we infer from Lemma \ref{lem5.1} that for any $N\ge2$, and 
some constant $C^*>0,$
\begin{equation}\label{L1}
\Upsilon_{s,t}^i
\le \frac{C^*\alpha_t}{\sqrt{N}}  \big(1+\E|\bar X_{s,s}^i|^2\big).
\end{equation}
Now, plugging  \eqref{L1} back into \eqref{L2} and taking advantage of  \eqref{**} yields 
 \begin{equation*} 
  \E|\Theta_{s,t}^{i,N}|
  \le \frac{C^{*}}{\sqrt{N}}\big(1+\E|\bar X_{s,s}^i|^2\big)+\frac{2 K_2\alpha_t}{c_1}\E f(|\bar Z_{s,t}^{i,N}|),\qquad t\ge s.
 \end{equation*} 
 Thus, substituting the estimate above into \eqref{ER*} implies that 
 \begin{equation} 
\begin{split}
\e^{c_*\int_s^t\alpha_u\,\d u}\E f(|\bar Z_{s,t}^{i,N}|) 
 &\le \E f(|\bar Z_{s,s}^{i,N}|)+\frac{2K_2}{c_1}\int_s^t\alpha_u\e^{c_*\int_s^u\alpha_r\,\d r}\E f(|\bar Z_{s,u}^{i,N}|)  \,\d u\\
 &\quad+ \frac{ C_*(1+c_1)}{\sqrt{N}}(1+\E|\bar X_{s,s}^i|^2)\int_s^t\alpha_u\e^{c_*\int_s^u\alpha_r\,\d r}\,\d u.
\end{split}
\end{equation}
Apparently, there exists a constant $K_2^*>0$ such that $c_*>2K_2/c_1$ for any $K_2\in(0,K_2^*].$ As a result, with the help of $\int_0^\tau\alpha_u\,\d u\in(0,\8)$ and by choosing $(\bar X_{s,s}^{i}, \bar X_{s,s}^{i,N})$ such that $\E|\bar X_{s,s}^{i}-\bar X_{s,s}^{i,N}|=\mathbb W_1  ( \mathscr L_{X_{s,s}^i}  , \mathscr L_{X_{s,s}^{i,N}}  )$
, the Gronwall inequality yields 
the assertion \eqref{L4} for any $K_2\in(0,K_2^*].$ 
\end{proof}

\begin{remark}
\label{remark on couplings}
In terms of the construction of  \eqref{P1--}, we observe  that, concerning the additive noise, we adopt the reflection coupling when the distance between the marginal  processes is greater than $\vv,$ employ the synchronous coupling provided  that the distance between the marginal processes is smaller than $\vv/2$, and exploit, in between,  the mixture of the reflection coupling and the synchronous coupling.  With regarding to the multiplicative noise, we  utilise all along  the synchronous coupling. 

In Lemma \ref{lem5.3}, we achieve the optimal decay rate (i.e., $N^{-\frac{1}{2}}$), which is independent of the dimension $d,$ w.r.t. the particle number once one part of the drift part is written as a convolution type and the initial distribution of IPS enjoy finite second moments.  Indeed, for the issue on the uniform-in-time PoC, the McKean-Vlasov SDEs can be much more general by taking advantage of \cite[Theorem 1]{fournier2015rate}. For such setting, the convergence rate w.r.t. the particle number is dependent on the dimension $d$ and becomes more and more worse when the dimension increases. Based on this point of view, for the drift term we prefer the convolution type rather than the general counterpart. 
\end{remark}

With the previous lemmas at hand, we are in position to complete the proofs of Theorems \ref{thm*}, \ref{thm4} and \ref{thm6} . First of all, we complete the proof of Theorem \ref{thm*}.

\begin{proof}[Proof of Theorem \ref{thm*}]
To complete the proof of Theorem \ref{thm*} we follow the approach of \cite[Proposition 2.1]
{bao2022random} which involves establishing: (a) semi-flow property; (b) time-shift equals omega-shift; (c) $L^1$-Wasserstein contraction of $(X_t)_{t\ge s}$ solving \eqref{*X}; (d) uniform $L^2$-moment estimates in an infinite horizon, respectively. 
Obviously, 
(a) and (b) follow directly from Lemma \ref{lem3} and Lemma  \ref{lem1} respectively. 
(c) is available by invoking Lemma \ref{lem5.2}. 
Lastly, (d) is verified by appealing to Lemma \ref{lem5.1}.
    This concludes the proof.
\color{black}
\end{proof}

Next, we intend to finish the proof of Theorem \ref{thm4}.
\begin{proof}[Proof of Theorem \ref{thm4}]
For ${\bf x}^N:=\big(x_1,\cdots,x_N\big)\in\R^{dN}$ and $t\in\R$, 
let 
\begin{equation*}
  {\bf b}_t({\bf x}^N) =\big(\,\bar b_t(x_1,{\bf x}^N),\cdots,\bar b_t(x_N,{\bf x}^N)\big) \quad \mbox{ and } \quad  {\bm{ \sigma}}_t({\bf x}^N\big) =\mbox{diag}\big(\hat \sigma_t(x_1),\cdots,\hat \sigma_t(x_N)\big),
\end{equation*}
where
$$  b_t(x_i,{\bf x}^N):=\hat b_t(x_i)+\frac{1}{N}\sum_{j=1}^N\tilde b_t(x_i,x_j).$$
Then, by invoking \eqref{X3} and \eqref{XX}, we deduce that 
\begin{equation*} 
\begin{split}
\nonumber
\<{\bf x}^N-{\bf y}^N, {\bf b}_t({\bf x}^N)- {\bf b}_t({\bf y}^N)\>
 +\frac{1}{2}\|{\bm{\sigma}}_t({\bf x}^N)-{\bm{\sigma}}_t({\bf y}^N)\|_{\rm HS}^2
&\le \alpha_t (K_0+K_1) \sum_{i=1}^N |x_i-y_i|^2\I_{\{|x_i-y_i|\le\ell_0\}}\\
&\quad-(K_1-2K_2)\alpha_t|{\bf x}^N-{\bf y}^N|^2.
\end{split}
\end{equation*}
Then, by applying It\^o's formula, for all $\xi^N=(\xi_1,\cdots,\xi_N)\in L^2(\OO\to(\R^{d})^N,\F_s,\P)$, there exists a constant $C_0(\xi)>0$ such that 
\begin{equation}\label{p*}
\sup_{t\ge s}\mathbb W_1\big(\mathscr L_{X_{s,t}^{i,N}},\delta_{{\bf0}}\big)\le C_0(\xi)  , 
\end{equation}
once $K_1>2K_2$. 
Next, by repeating the argument of Lemma \ref{lem5.2}, we find that there exists a constant $C^*,\ll,K_2^*>0$ such that for all $K_2\in(0,K_2^*]$,
\begin{equation}\label{p**}
   \mathbb W_1\Big(\mathscr L_{{\bf X}^{\xi^N}_{s,t}},\mathscr L_{{\bf X}^{\eta^N}_{s,t}}\Big) \le C^*\e^{-\ll(t-s)}\mathbb W_1\big(\mathscr L_{ \xi^N },\mathscr L_{ \eta^N }\big).
\end{equation}
Thus, combining \eqref{p*} with \eqref{p**} and applying \cite[Proposition 2.1]{bao2022random}, we finish the proof of Theorem \ref{thm4}.
\end{proof}

At last, we address the proof of Theorem \ref{thm6}.
\begin{proof}[Proof of Theorem \ref{thm6}]
The proof of Theorem \ref{thm6}
is complete by invoking the triangle inequality and taking Theorems \ref{thm*} and \ref{thm4}
as well as Lemma \ref{lem5.3} into consideration. 
\end{proof}

\section*{Acknowledgements}

G. dos Reis acknowledges partial support from the FCT – Fundação para a Ciência e a Tecnologia, I.P., under the scope of the projects UIDB/00297/2020 (\url{https://doi.org/10.54499/UIDB/00297/2020}) and UIDP/00297/2020 (\url{https://doi.org/10.54499/UIDP/00297/2020}) (Center for Mathematics and Applications, NOVA Math).

The research of Jianhai Bao is supported by the National Key R\& D Program of China (2022YFA1006004) and 
the National Natural Science Foundation of China (No. 12071340). 

Yue Wu acknowledges support from the Lower Saxony – Scotland Tandem Fellowship Programme funded by the Ministry of Science and Culture of Lower Saxony (MWK).







\begin{thebibliography}{10}

\bibitem{AdamsSalkeld2022LDPExitTimes}
D.~Adams, G.~dos Reis, R.~Ravaille, W.~Salkeld, and J.~Tugaut.
\newblock Large deviations and exit-times for reflected {M}c{K}ean-{V}lasov
  equations with self-stabilising terms and superlinear drifts.
\newblock {\em Stochastic Processes and their Applications}, 146:264--310,
  2022.

\bibitem{Alasseur2020storageSmartGrids}
C.~Alasseur, I.~Ben~Taher, and A.~Matoussi.
\newblock An extended mean field game for storage in smart grids.
\newblock {\em Journal of Optimization Theory and Applications},
  184(2):644--670, 2020.

\bibitem{arnold1995random}
L.~Arnold, C.~K. Jones, K.~Mischaikow, G.~Raugel, and L.~Arnold.
\newblock {\em Random dynamical systems}.
\newblock Springer, 1995.

\bibitem{bao2022random}
J.~Bao and Y.~Wu.
\newblock Random periodic solutions for stochastic differential equations with
  non-uniform dissipativity.
\newblock {\em arXiv preprint arXiv:2202.09771}, 2022.

\bibitem{bottou2018optimization}
L.~Bottou, F.~E. Curtis, and J.~Nocedal.
\newblock Optimization methods for large-scale machine learning.
\newblock {\em SIAM review}, 60(2):223--311, 2018.

\bibitem{carmona2018probabilistic}
R.~Carmona and F.~Delarue.
\newblock {\em Probabilistic Theory of Mean Field Games with Applications
  I-II}.
\newblock Springer, 2018.

\bibitem{cattiaux2014semi}
P.~Cattiaux and A.~Guillin.
\newblock Semi log-concave markov diffusions.
\newblock {\em S{\'e}minaire de probabilit{\'e}s XLVI}, pages 231--292, 2014.

\bibitem{chapman2012coarsening}
C.~K. Chapman.
\newblock {\em Coarsening dynamical systems: dynamic scaling, universality and
  mean-field theories}.
\newblock PhD thesis, University of Glasgow, 2012.

\bibitem{chassaing2011non}
P.~Chassaing and J.~Mairesse.
\newblock A non-ergodic probabilistic cellular automaton with a unique
  invariant measure.
\newblock {\em Stochastic Processes and their Applications},
  121(11):2474--2487, 2011.

\bibitem{chekroun2011stochastic}
M.~D. Chekroun, E.~Simonnet, and M.~Ghil.
\newblock Stochastic climate dynamics: Random attractors and time-dependent
  invariant measures.
\newblock {\em Physica D: Nonlinear Phenomena}, 240(21):1685--1700, 2011.

\bibitem{chen2004markov}
M.~Chen.
\newblock {\em From Markov chains to non-equilibrium particle systems}.
\newblock World scientific, 2004.

\bibitem{chen2023wellposedness}
X.~Chen, G.~dos Reis, and W.~Stockinger.
\newblock Wellposedness, exponential ergodicity and numerical approximation of
  fully super-linear {McK}ean--{V}lasov {SDE}s and associated particle systems.
\newblock {\em arXiv preprint arXiv:2302.05133}, 2023.

\bibitem{Chen2024SuperSUper}
X.~Chen and G.~c. dos Reis.
\newblock Euler simulation of interacting particle systems and
  {M}c{K}ean-{V}lasov {SDE}s with fully super-linear growth drifts in space and
  interaction.
\newblock {\em IMA J. Numer. Anal.}, 44(2):751--796, 2024.

\bibitem{chen2024mean}
X.~Chen, W.~Kuang, L.~Deng, W.~Han, B.~Bai, and G.~dos Reis.
\newblock A mean field ansatz for zero-shot weight transfer.
\newblock {\em arXiv preprint arXiv:2408.08681}, 2024.

\bibitem{cormier2021hopf}
Q.~Cormier, E.~Tanr{\'e}, and R.~Veltz.
\newblock Hopf bifurcation in a mean-field model of spiking neurons.
\newblock {\em Electronic Journal of Probability}, 26:1--40, 2021.

\bibitem{da1996ergodicity}
G.~Da~Prato and J.~Zabczyk.
\newblock {\em Ergodicity for infinite dimensional systems}, volume 229.
\newblock Cambridge University Press, 1996.

\bibitem{dai2020oscillatory}
P.~Dai~Pra, M.~Formentin, and G.~Pelino.
\newblock Oscillatory behavior in a model of non-markovian mean field
  interacting spins.
\newblock {\em Journal of Statistical Physics}, 179:690--712, 2020.

\bibitem{reis2018simulation}
G.~dos Reis, S.~Engelhardt, and G.~Smith.
\newblock Simulation of {M}c{K}ean-{V}lasov {SDE}s with super-linear growth.
\newblock {\em IMA Journal of Numerical Analysis}, 42(1):874--922, 2022.

\bibitem{salkeld2019LDPMV-SDEs}
G.~dos Reis, W.~Salkeld, and J.~Tugaut.
\newblock Freidlin-{W}entzell {LDP} in path space for {M}c{K}ean-{V}lasov
  equations and the functional iterated logarithm law.
\newblock {\em The Annals of Applied Probability}, 29(3):1487--1540, 2019.

\bibitem{durmus2020elementary}
A.~Durmus, A.~Eberle, A.~Guillin, and R.~Zimmer.
\newblock An elementary approach to uniform in time propagation of chaos.
\newblock {\em Proceedings of the American Mathematical Society},
  148(12):5387--5398, 2020.

\bibitem{eberle2016reflection}
A.~Eberle.
\newblock Reflection couplings and contraction rates for diffusions.
\newblock {\em Probability Theory and Related Fields}, 166(3):851--886, 2016.

\bibitem{feng2017numerical}
C.~Feng, Y.~Liu, and H.~Zhao.
\newblock Numerical approximation of random periodic solutions of stochastic
  differential equations.
\newblock {\em Zeitschrift f{\"u}r angewandte Mathematik und Physik},
  68(5):1--32, 2017.

\bibitem{feng2016anticipating}
C.~Feng, Y.~Wu, and H.~Zhao.
\newblock Anticipating random periodic solutions---{I}. {SDE}s with
  multiplicative linear noise.
\newblock {\em Journal of Functional Analysis}, 271(2):365--417, 2016.

\bibitem{wu2018random}
C.~Feng, Y.~Wu, and H.~Zhao.
\newblock Anticipating random periodic solutions--{II}. {SPDE}s with
  multiplicative linear noise.
\newblock {\em arXiv preprint arXiv:1803.00503}, 2018.

\bibitem{feng2012random}
C.~Feng and H.~Zhao.
\newblock Random periodic solutions of {SPDE}s via integral equations and
  {W}iener-{S}obolev compact embedding.
\newblock {\em Journal of Functional Analysis}, 262(10):4377--4422, 2012.

\bibitem{feng2020random}
C.~Feng and H.~Zhao.
\newblock Random periodic processes, periodic measures and ergodicity.
\newblock {\em Journal of Differential Equations}, 269(9):7382--7428, 2020.

\bibitem{feng2011pathwise}
C.~Feng, H.~Zhao, and B.~Zhou.
\newblock Pathwise random periodic solutions of stochastic differential
  equations.
\newblock {\em Journal of Differential Equations}, 251(1):119--149, 2011.

\bibitem{fournier2015rate}
N.~Fournier and A.~Guillin.
\newblock On the rate of convergence in {W}asserstein distance of the empirical
  measure.
\newblock {\em Probability Theory and Related Fields}, 162(3-4):707--738, 2015.

\bibitem{FriedrichKinzel2009}
J.~Friedrich and W.~Kinzel.
\newblock Dynamics of recurrent neural networks with delayed unreliable
  synapses: metastable clustering.
\newblock {\em Journal of Computational Neuroscience}, 27(1):65--80, 2009.

\bibitem{guillin2022convergence}
A.~Guillin, P.~Le~Bris, and P.~Monmarch\'e.
\newblock Convergence rates for the {V}lasov-{F}okker-{P}lanck equation and
  uniform in time propagation of chaos in non convex cases.
\newblock {\em Electronic Journal of Probability}, 27:Paper No. 124, 44, 2022.

\bibitem{hairer2006ergodic}
M.~Hairer.
\newblock Ergodic properties of {M}arkov processes.
\newblock {\em Lecture notes}, 2006.

\bibitem{HIPP2014Resonance}
S.~Herrmann, P.~Imkeller, I.~Pavlyukevich, and D.~Peithmann.
\newblock {\em Stochastic resonance}, volume 194 of {\em Mathematical Surveys
  and Monographs}.
\newblock American Mathematical Society, Providence, RI, 2014.
\newblock A mathematical approach in the small noise limit.

\bibitem{HerrmannTugaut2010}
S.~Herrmann and J.~Tugaut.
\newblock Non-uniqueness of stationary measures for self-stabilizing processes.
\newblock {\em Stochastic Processes and their Applications}, 120(7):1215--1246,
  2010.

\bibitem{huang2021distribution}
X.~Huang, P.~Ren, and F.-Y. Wang.
\newblock Distribution dependent stochastic differential equations.
\newblock {\em Frontiers of Mathematics in China}, 16(2):257--301, 2021.

\bibitem{kolokoltsov2010nonlinear}
V.~N. Kolokoltsov.
\newblock {\em Nonlinear Markov processes and kinetic equations}, volume 182.
\newblock Cambridge University Press, 2010.

\bibitem{lacker2023sharp}
D.~Lacker and L.~Le~Flem.
\newblock Sharp uniform-in-time propagation of chaos.
\newblock {\em Probability Theory and Related Fields}, pages 1--38, 2023.

\bibitem{LocherbachMonmarche2022Metastability}
E.~L\"{o}cherbach and P.~Monmarch\'{e}.
\newblock Metastability for systems of interacting neurons.
\newblock {\em Annales de l'Institut Henri Poincar\'{e} Probabilit\'{e}s et
  Statistiques}, 58(1):343--378, 2022.

\bibitem{luccon2020emergence}
E.~Lu{\c{c}}on and C.~Poquet.
\newblock Emergence of oscillatory behaviors for excitable systems with noise
  and mean-field interaction: a slow-fast dynamics approach.
\newblock {\em Communications in Mathematical Physics}, 373:907--969, 2020.

\bibitem{luo2019refined}
D.~Luo and J.~Wang.
\newblock Refined basic couplings and {W}asserstein-type distances for {SDE}s
  with {L}\'evy noises.
\newblock {\em Stochastic Processes and their Applications}, 129(9):3129--3173,
  2019.

\bibitem{majka2020nonasymptotic}
M.~B. Majka, A.~Mijatovi\'c, and L.~u. Szpruch.
\newblock Nonasymptotic bounds for sampling algorithms without log-concavity.
\newblock {\em The Annals of Applied Probability}, 30(4):1534--1581, 2020.

\bibitem{Marini2023etalfrustratednetworks}
E.~Marini, L.~Andreis, F.~Collet, and M.~Formentin.
\newblock Noise-induced periodicity in a frustrated network of interacting
  diffusions.
\newblock {\em NoDEA Nonlinear Differential Equations and Applications},
  30(3):Paper No. 34, 35, 2023.

\bibitem{Meleard1996}
S.~M{\'e}l{\'e}ard.
\newblock {A}symptotic behaviour of some interacting particle systems;
  {M}c{K}ean-{V}lasov and {B}oltzmann models.
\newblock In {\em Probabilistic models for nonlinear partial differential
  equations}, pages 42--95. Springer, 1996.

\bibitem{prevot2007concise}
C.~Pr{\'e}v{\^o}t and M.~R{\"o}ckner.
\newblock {\em A Concise Course on Stochastic Partial Differential Equations},
  volume 1905.
\newblock Springer, 2007.

\bibitem{ren2021exponential}
P.~Ren, K.-T. Sturm, and F.-Y. Wang.
\newblock Exponential ergodicity for time-periodic {McK}ean-{V}lasov {SDE}s.
\newblock {\em arXiv preprint arXiv:2110.06473}, 2021.

\bibitem{scheutzow1985some}
M.~Scheutzow.
\newblock Some examples of nonlinear diffusion processes having a time-periodic
  law.
\newblock {\em The Annals of Probability}, pages 379--384, 1985.

\bibitem{scheutzow1986periodic}
M.~Scheutzow.
\newblock Periodic behavior of the stochastic brusselator in the mean-field
  limit.
\newblock {\em Probability Theory and Related Fields}, 72(3):425--462, 1986.

\bibitem{sznitman1991topics}
A.-S. Sznitman.
\newblock Topics in propagation of chaos.
\newblock In {\em Ecole d'{\'e}t{\'e} de probabilit{\'e}s de Saint-Flour
  XIX—1989}, pages 165--251. Springer, 1991.

\bibitem{van2014periodic}
L.~van Veen and K.~R. Green.
\newblock Periodic solutions to a mean-field model for electrocortical
  activity.
\newblock {\em The European Physical Journal Special Topics},
  223(13):2979--2988, 2014.

\bibitem{wang2018distribution}
F.-Y. Wang.
\newblock Distribution dependent {SDE}s for {L}andau type equations.
\newblock {\em Stochastic Processes and their Applications}, 128(2):595--621,
  2018.

\bibitem{wang2021exponential}
F.-Y. Wang.
\newblock Exponential ergodicity for non-dissipative {M}c{K}ean-{V}lasov
  {SDE}s.
\newblock {\em Bernoulli}, 29(2):1035--1062, 2023.

\bibitem{wu2021random}
Y.~Wu.
\newblock Backward {E}uler--{M}aruyama method for the random periodic solution
  of a stochastic differential equation with a monotone drift.
\newblock {\em Journal of Theoretical Probability}, pages 1--18, 2022.

\bibitem{wu2021galerkin}
Y.~Wu and C.~Yuan.
\newblock The {G}alerkin analysis for the random periodic solution of
  semilinear stochastic evolution equations.
\newblock {\em Journal of Theoretical Probability}, 37(1):133--159, 2024.

\bibitem{zhao2009random}
H.~Zhao and Z.-H. Zheng.
\newblock Random periodic solutions of random dynamical systems.
\newblock {\em Journal of Differential equations}, 246(5):2020--2038, 2009.

\end{thebibliography}

\end{document}